\documentclass[12pt]{amsart}

\usepackage[latin1]{inputenc}

\usepackage{color,enumitem}

\newtheorem{theo}{Theorem}[section]
\newtheorem{lem}[theo]{Lemma}
\newtheorem{theorem}[theo]{Theorem}
\newtheorem{proposition}[theo]{Proposition}
\newtheorem{prop}[theo]{Proposition}

\newtheorem{defi}[theo]{Definition}
\newtheorem{remark}[theo]{Remark}

\newtheorem{example}[theo]{Example}

\newcommand{\He}{\mathbb H}
\newcommand{\R}{\mathbb R}
\newcommand{\N}{\mathbb N}
\newcommand{\eps}{\varepsilon}

\newcommand{\SPAN}{\mathrm{Span}}

\begin{document}

\title[Whitney Theorem for Curves in Carnot Groups]{Pliability, or the Whitney Extension Theorem for Curves in Carnot Groups}

\begin{abstract}
The Whitney extension theorem is a classical result in analysis giving a necessary and sufficient condition for a function defined on a closed set to be extendable to the whole space with a given class of regularity. It has been adapted to several settings, among which the one of Carnot groups. However, the target space has generally been assumed to be equal to $\R^d$ for some $d\ge 1$.

We focus here on the extendability problem for general ordered pairs $(G_1,G_2)$ (with $G_2$ non-Abelian). We
analyze in particular the case $G_1=\R$ and characterize the groups $G_2$ for which the Whitney extension property holds, in terms of a newly  introduced notion that we call \emph{pliability}. Pliability happens to be related to rigidity as defined by Bryant an Hsu. We exploit this relation in order to provide examples of non-pliable Carnot groups, that is, Carnot groups so that the Whitney extension property does
not hold. We use  geometric control theory results on the accessibility of control affine systems in order to test
the pliability of a Carnot group. In particular, we recover some recent results by Le Donne, Speight and Zimmermann about
Lusin approximation in Carnot groups of step 2 and Whitney extension in Heisenberg groups. We extend such results to all pliable Carnot
groups, and we show that the latter may be of arbitrarily large step.

\end{abstract}


\author{Nicolas Juillet} \author{Mario Sigalotti}
\address{Institut de Recherche Math\'ematique Avanc\'ee\\
  UMR 7501\\
 Universit\'e de Strasbourg et CNRS\\
 7 rue Ren\'e  Descartes\\
 67\,000 Strasbourg\\
 France}
\email{nicolas.juillet@math.unistra.fr}
 \address{
 Inria, team GECO \& CMAP, \'Ecole Polytechnique, CNRS, Universit\'e Paris-Saclay\\ Palaiseau\\ France} 
\email{mario.sigalotti@inria.fr}

\maketitle

\section{Introduction}

Extending functions is a basic but fundamental tool in analysis. Fundamental is in particular the extension theorem established by 
H.~Whitney in 1934, which guarantees the existence of an extension of a function defined on a closed set of a finite-dimensional vector space to a function of class $\mathcal{C}^k$, provided that the minimal obstruction imposed by Taylor series is satisfied. The Whitney extension theorem plays a significative part in the study of ideals of differentiable functions (see \cite{Malgrange}) and its variants are still an active research topic of classical analysis 
(see for instance \cite{Fef}).

%\medskip

Analysis on Carnot groups with a homogeneous distance like the Carnot--Carath\'eodory distance, as presented in Folland and Stein's monograph \cite{FoS}, is nowadays a classical topic too. 
Carnot groups provide a generalization of finite-dimensional vector spaces that is both close to the original model and radically different. 
This is why Carnot groups provide a wonderful field of investigation in many branches of mathematics. 
Not only the setting is elegant and rich but it is at the natural crossroad between different fields of mathematics, as for instance analysis of PDEs or geometric control theory (see for instance \cite{trimestre} for a contemporary account). It is therefore natural to recast the Whitney extension theorem in the context of Carnot groups. As far as we know, the first generalization of a Whitney extension theorem to Carnot groups can be found in \cite{FSS,FSS2}, where 
De Giorgi's result on sets of finite perimeter is adapted 
first to the Heisenberg group and then 
to any Carnot group of step 2. 
This generalization is used in \cite{KS}, where the authors stress the difference between intrinsic regular hypersurfaces and classical $\mathcal{C}^1$ hypersurfaces in the Heisenberg group. The recent paper \cite{VP} gives a final statement for the Whitney extension theorem for scalar-valued functions on Carnot groups: The most natural generalization that one can imagine holds in its full strength (for more details, see Section~\ref{s:2}).

%\medskip

The study of the Whitney extension property for Carnot groups is however not closed. 
Following a suggestion by Serra Cassano in \cite{serra-cassano-notes}, one might consider maps between Carnot groups instead of 
solely scalar-valued functions on Carnot groups.  The new question presents richer geometrical features and echoes classical topics of metric geometry. We think in particular of the classification of Lipschitz embeddings for metric spaces and of the related question of the extension of Lipschitz maps between metric spaces. 
We refer to \cite{BF,WY,RW,BLP} for the corresponding results for 
the most usual Carnot groups: Abelian groups $\R^m$ or Heisenberg groups $\He_n$ (of topological dimension $2n+1$). In view of Pansu--Rademacher theorem on Lipschitz maps (see Theorem~\ref{PR_them}), the most directly related Whitney extension problem is the one for $\mathcal{C}^1_H$-maps, the so-called horizontal maps of class $\mathcal{C}^1$ defined on Carnot groups. This is the framework of our paper.

%\medskip

Simple pieces of argument show that the Whitney extension theorem does not generalize to every ordered pair of Carnot groups. Basic facts in 
contact geometry suggest 
that \emph{the extension does not hold for $(\R^{n+1},\He_n)$}, i.e., for maps from $\R^{n+1}$ to $\He_n$. It is actually known that local algebraic constraints of first order make $n$ the maximal dimension for a Legendrian submanifold in a contact manifold of dimension $2n+1$. In fact if the derivative of a differentiable map has range in the kernel of the contact form, the range of the map has dimension at most $n$. A map from $\R^{n+1}$ to $\He_n$ is $\mathcal{C}^1_H$ if it is $\mathcal{C}^1$ with horizontal derivatives, i.e., if its derivatives take value in the kernel of the canonical contact form. In particular, a $\mathcal{C}^1_H$-map defined on $\R^{n+1}$ is nowhere of maximal rank. Moreover, it is a consequence of the Pansu--Rademacher theorem that a Lipschitz map from $\R^{n+1}$ to $\He_n$ is derivable  at almost every point with only horizontal derivatives. Again $n$ is their maximal rank. In order to contradict the extendability of Lipschitz maps,  it is enough to define a function on a subset whose topological constraints force any possible extension 
to have maximal rank at some point. \emph{Let us sketch a concrete example that provides a constraint for the Lipschitz extension problem}: It is known that $\R^n$ can be isometrically embedded in $\He_n$ with the exponential map (for the Euclidean and Carnot--Carath\'eodory distances). One can also consider two `parallel' copies of $\R^n$ in $\R^{n+1}$ mapped to parallel images in $\He_n$: the second is obtained from the first by a vertical translation. Aiming for a contradiction, suppose that there exists an extending Lipschitz map $F$. It provides on $\R^n\times [0,1]$ a Lipschitz homotopy between $F(\R^{n}\times \{0\})$ and $F(\R^{n}\times \{1\})$. Using the definition of a Lipschitz map and some topology, the topological dimension of the range is at least $n+1$ and its $(n+1)$-Hausdorff measure is positive. This is not possible because of the dimensional constraints explained above. See \cite{BF} for a more rigorous proof 
using a different set as a domain for the function to be extended. The proof in \cite{BF} is formulated in terms of index theory 
and purely $(n+1)$-unrectifiability of $\He_n$. The latter property means that the $(n+1)$-Hausdorff measure of the range of a Lipschitz map is zero. Probably this construction and some other ideas from the works on the Lipschitz extension problem \cite{BF,WY,RW,BLP} can be adapted to the Whitney extension problem. It is not really our concern in the present article to list the similarities between the two problems, but rather to exhibit a class of ordered pairs of Carnot groups for which the validity of the Whitney extension problem depends on the geometry of the groups. Note that a different type of counterexample to the Whitney extension theorem, involving groups which are neither Euclidean spaces nor Heisenberg groups, has been obtained by A.~Khozhevnikov in \cite{Koz}. It is described in Example \ref{ex:Koz}. 

%\medskip

Our work is motivated by  
F.~Serra Cassano's suggestion in his Paris' lecture notes at the Institut Henri Poincar\'e in 2014 \cite{serra-cassano-notes}. He proposes --- (i) to choose general Carnot groups $\mathbb{G}$ as target space, (ii) to look at $\mathcal{C}^1_H$ curves only, i.e., $\mathcal{C}^1$ maps from $\R$ to $\mathbb{G}$ with horizontal derivatives. As we will see, the problem is very different from the Lipschitz extension problem for $(\R,\mathbb{G})$ and from the Whitney extension problem for $(\mathbb{G},\R)$. Indeed, both such problems can be solved for every $\mathbb{G}$, while the 
answer to the extendibility question asked by Serra Cassano depends on the choice of $\mathbb{G}$. 
More precisely, we provide a geometric characterization of those $\mathbb{G}$ for which the $\mathcal{C}^1_H$-Whitney extension problem for $(\R,\mathbb{G})$ can always be solved. We say in this case that the pair $(\mathbb{R},\mathbb{G})$ has the $\mathcal{C}^1_H$ extension property.
Examples of target non-Abelian
Carnot groups 
for which 
$\mathcal{C}^1_H$ extendibility is possible 
have been identified by S.~Zimmerman in \cite{zimmerman}, where it is proved that
for every  $n\in \N$ the pair
$(\mathbb{R},\He_n)$
has the $\mathcal{C}^1_H$ extension property.

The main component of the characterization of Carnot groups $\mathbb{G}$ for which  $(\mathbb{R},\mathbb{G})$ has the $\mathcal{C}^1_H$ extension property is the notion of \emph{pliable horizontal vector}. A horizontal vector $X$ (identified with a left-invariant vector field) is pliable if for every $p\in\mathbb{G}$ and every neighborhood $\Omega$ of $X$ in the horizontal layer of $\mathbb{G}$, the support of all $\mathcal{C}^1_H$ curves with derivative in $\Omega$ starting from $p$ in the direction $X$ form a neighborhood of the integral curve  of $X$ starting from $p$ (for details, see Definition~\ref{d:pliable} and Proposition~\ref{p:pliableiffweaklypliable}). 
This notion is close but not equivalent to the property of the integral curves of $X$ not  to be rigid in the 
sense introduced by Bryant and Hsu in \cite{BH}, as we illustrate by an example (Example~\ref{e:superEngel}). 
We say that \emph{a Carnot group $\mathbb{G}$ is pliable if all its horizontal vectors are pliable}. Since any rigid integral curve of a horizontal vector $X$ is not pliable, it is not hard to show that there exist non-pliable Carnot groups of any dimension larger than $3$ and of any step larger than $2$ (see Example~\ref{e:Engel}).  
On the other hand, we give some criteria ensuring the pliability of a Carnot group, notably the fact that it has step 2 (Theorem~\ref{t:step2}). 
We also prove the existence of pliable groups of any positive step (Proposition~\ref{p:every_step}). 

Our main theorem is the following. 
\begin{theo}\label{thm-necessary-and-sufficient}
 The pair  $(\mathbb{R},\mathbb{G})$ has the $\mathcal{C}^1_H$ extension property
 if and only if 
 $\mathbb{G}$ is pliable. 
\end{theo}
 
The paper is organized as follows:  in Section~\ref{s:2} we recall some basic facts about Carnot groups and we present the $\mathcal{C}^1_H$-Whitney condition in the light of the Pansu--Rademacher Theorem. In Section~\ref{sec:necessary} we introduce the notion of pliability, we discuss its relation with rigidity, and we show that pliability of $\mathbb{G}$ is necessary 
for the $\mathcal{C}^1_H$ extension property to hold for  $(\mathbb{R},\mathbb{G})$ (Theorem~\ref{thm-necessary}). The proof of this result goes by assuming that a non-pliable horizontal vector exists and using it to provide an explicit construction of a $\mathcal{C}^1_H$ map defined on a closed subset of $\mathbb{R}$ which cannot be extended on $\mathbb{R}$. 
 Section~\ref{sec:sufficient} is devoted to proving that pliability is also a sufficient condition (Theorem~\ref{thm-sufficient}). In Section~\ref{s:5} we use our result to extend some Lusin-like theorem proved recently by G.~Speight for Heisenberg groups \cite{Spe} (see also \cite{zimmerman} for an alternative proof). More precisely, it is proved in \cite{LeDonne-Speight} that an absolutely continuous curve in a group of step 2 coincides on a set of arbitrarily small complement with a $\mathcal{C}^1_H$ curve. We show that this is the case for pliable Carnot groups (Proposition~\ref{our_Lusin_appr}). Finally, in Section~\ref{s:6} we give some criteria for testing the pliability of a Carnot group. We first show that the zero horizontal vector is always pliable (Proposition~\ref{p:zero-pliable}). Then, by applying some results of control theory providing criteria under which  the endpoint mapping is open, we show
that $\mathbb{G}$ is pliable if its step is equal to 2.

\section{Whitney condition in Carnot groups}\label{s:2}

A  nilpotent Lie group $\mathbb{G}$ is said to be a \emph{Carnot group} if it is stratified, in the sense that 
its Lie algebra $\mathfrak{G}$ admits a direct sum decomposition 
$$\mathfrak{G}_1\oplus\cdots\oplus \mathfrak{G}_s,$$
called \emph{stratification}, such that  $[\mathfrak{G}_i,\mathfrak{G}_j]=\mathfrak{G}_{i+j}$ for every $i,j\in \N^*$ with $i+j\leq s$
and $[\mathfrak{G}_i,\mathfrak{G}_j]=\{0\}$ if $i+j>s$. 
We recall that $[\mathfrak{G}_i,\mathfrak{G}_j]$ denotes the linear space spanned by $\{[X,Y]\in\mathfrak{G}\mid X\in \mathfrak{G}_i,\;Y\in \mathfrak{G}_j\}$. 
The subspace $\mathfrak{G}_1$ is called the \emph{horizontal layer} and it is also denoted by $\mathfrak{G}_H$. We say that $s$ is the \emph{step} of $\mathbb{G}$ if $\mathfrak{G}_s\ne\{0\}$. 
The group product of two elements $x_1,x_2\in \mathbb{G}$ is denoted by $x_1\cdot x_2$.
Given $X\in \mathfrak{G}$ we write $\mathrm{ad}_X:\mathfrak{G}\to \mathfrak{G}$ for the operator defined by $\mathrm{ad}_X Y=[X,Y]$.

The Lie algebra $\mathfrak{G}$ can be identified with the family of left-invariant vector fields on $\mathbb{G}$. The exponential is the application that maps a vector $X$ of $\mathfrak{G}$ into the end-point at time $1$ of the integral curve of the vector field $X$ starting from the identity of $\mathbb{G}$, denoted by $0_\mathbb{G}$. That is, if
$$\left\{
\begin{aligned}
\gamma(0)&=0_{\mathbb{G}}\\
\dot{\gamma}(t)&=X\circ\gamma(t)
\end{aligned}
\right.
$$
then $\gamma(1)=\exp(X)$. We also denote by $e^{tX}:\mathbb{G}\to \mathbb{G}$ the flow of the left-invariant vector field $X$ at time $t$. Notice that $e^{tX}(p)=p\cdot \exp(tX)$. Integral curves of left-invariant vector fields are 
said to be \emph{straight curves}.

The Lie group $\mathbb{G}$ is diffeomorphic to $\R^N$ with $N=\sum_{k=1}^s\dim(\mathfrak{G}_k)$. 
A usual way to identify $\mathbb{G}$ and $\R^N$ through a global system of coordinates is to pull-back by $\exp$ the group structure from $\mathbb{G}$ to $\mathfrak{G}$, where it can be expressed by the Baker--Campbell--Hausdorff formula. In this way $\exp$ becomes a mapping of $\mathfrak{G}=\mathbb{G}$ onto itself that is simply the identity.

For any $\lambda\in \R$ we introduce the dilation $\Delta_\lambda:\mathfrak{G}\to \mathfrak{G}$ uniquely characterized by
$$\left\{
\begin{aligned}
\Delta_{\lambda}([X,Y])&=[\Delta_\lambda(X),\Delta_\lambda (Y)]\quad \text{for any }X,Y\in\mathfrak{G},\\
\Delta_\lambda(X)&=\lambda X\quad\text{for any }X \in \mathfrak{G}_1.
\end{aligned}
\right.
$$
Using the decomposition $X=X_1+\cdots+X_s$ with $X_k\in \mathfrak{G}_k$, it holds $\Delta_\lambda(X)=\sum_{k=1}^s \lambda^k X_k$.
For any $\lambda\in \R$ we also define on $\mathbb{G}$ the dilation
$\delta_\lambda=\exp\circ\Delta_\lambda\circ\exp^{-1}$.

Given an absolutely continuous curve $\gamma:[a,b]\to \mathbb{G}$, the velocity $\dot \gamma(t)$, which exists from almost every $t\in [a,b]$, is identified with the element of  $\mathfrak{G}$ whose associated left-invariant vector field, evaluated at $\gamma(t)$, is equal to $\dot \gamma(t)$. An absolutely continuous curve $\gamma$ is said to be \emph{horizontal} if $\dot \gamma(t)\in \mathfrak{G}_H$ for almost every $t$.
For any interval $I$ of $\R$, we denote by  $\mathcal{C}^1_H(I,\mathbb{G})$ the space of all curves 
$\phi\in \mathcal{C}^1(I,\mathbb{G})$ such that  $\dot \phi(t)\in \mathfrak{G}_H$ for every $t\in I$.

Assume that the horizontal layer $\mathfrak{G}_H$ of the algebra is endowed with a  
quadratic norm $\|\cdot\|_{\mathfrak{G}_H}$. The Carnot--Carath\'eodory distance $d_{\mathbb{G}}(p,q)$ between two points $p,q\in \mathbb{G}$ is then defined as the  minimal length of a horizontal curve connecting $p$ and $q$, i.e., 
\begin{align*}
d_{\mathbb{G}}(p,q)=\inf\Big\{\int_a^b\|\dot{\gamma}(t)\|_{\mathfrak{G}_H}dt\mid &\gamma:[a,b]\to\mathbb{G}\mbox{ horizontal},\\
&\gamma(a)=p,\;\gamma(b)=q\Big\}.
\end{align*}
Note that $d_{\mathbb{G}}$ is left-invariant. 
It is known that $d_\mathbb{G}$ provides the same topology as the usual one on $\mathbb{G}$. 
Moreover it is homogeneous, i.e., $d_{\mathbb{G}}(\delta_\lambda p,\delta_\lambda q)=|\lambda|d_{\mathbb{G}}(p,q)$ for any $\lambda\in \R$.

Observe that the Carnot--Carath\'eodory distance depends on the norm $\|\cdot\|_{\mathfrak{G}_H}$ considered on $\mathfrak{G}_H$. However all Carnot--Carath\'eodory distances are in fact metrically equivalent. They are even equivalent with any left-invariant homogeneous distance \cite{FoS} in a very similar way as all norms on a finite-dimensional vector space are equivalent.

Notice that 
$d_{\mathbb{G}}(p,\cdot)$ can 
be seen as the value function of the optimal control problem
\begin{equation*}
\left\{
\begin{aligned}
&\dot\gamma=\sum_{i=1}^{m} u_i X_i(\gamma),\quad (u_1,\dots,u_m)\in\R^m,\\
&\gamma(a)=p,\\
&\int_a^{b}\sqrt{u_1(t)^2+\cdots+u_m(t)^2}dt\to \min,  
\end{aligned}
\right.
\end{equation*}
where $X_1,\dots,X_m$ is a $\|\cdot\|_{\mathfrak{G}_H}$-orthonormal basis of $\mathfrak{G}_H$. 

Finally,  the space $\mathcal{C}_H^1([a,b],\mathbb{G})$  of horizontal curves of class $\mathcal{C}^1$ can be endowed with a natural $\mathcal{C}^1$ metric associated with $(d_{\mathbb{G}},\|\cdot\|_{\mathfrak{G}_H})$ as follows:
the distance between two curves $\gamma_1$ and $\gamma_2$ in $\mathcal{C}_H^1([a,b],\mathbb{G})$ is 
$$\max\left(\sup_{t\in[a,b]}d_{\mathbb{G}}(\gamma_1(t),\gamma_2(t)),\sup_{t\in[a,b]}\|\dot{\gamma}_2(t)-\dot{\gamma}_1(t)\|_{\mathfrak{G}_H}\right).$$

In the following, we will write $\|\dot{\gamma}_2-\dot{\gamma}_1\|_{\infty,\mathfrak{G}_H}$ to denote the quantity $\sup_{t\in[a,b]}\|\dot{\gamma}_2(t)-\dot{\gamma}_1(t)\|_{\mathfrak{G}_H}$.

\subsection{Whitney condition}

A \emph{homogeneous homomorphism} between two Carnot groups $\mathbb{G}_1$ and $\mathbb{G}_2$ is a group morphism $L:\mathbb{G}_1\to\mathbb{G}_2$ with $L\circ\delta_\lambda^{\mathbb{G}_1}=\delta_\lambda^{\mathbb{G}_2}\circ L$ for any $\lambda\in\R$. Moreover $L$ is a homogeneous homomorphism if and only if $\exp^{-1}_{\mathbb{G}_2}\circ L\circ\exp_{\mathbb{G}_1}$ is a homogeneous Lie algebra morphism. It is in particular a linear map on $\mathbb{G}_1$ identified with $\mathfrak{G}_{\mathbb{G}_1}$. The first layer is mapped on the first layer so that a homogeneous homomorphism from $\R$ to $\mathbb{G}_2$ has the form $L(t)=\exp_{\mathbb{G}_2}(tX)$, where $X\in \mathfrak{G}^{\mathbb{G}_2}_{H}$.

\begin{prop}[Pansu--Rademacher Theorem]\label{PR_them}
Let $f$ be a locally Lipschitz map from an open subset $U$ of $ \mathbb{G}_1$ into $\mathbb{G}_2$. Then for almost every $p\in U$, there exists a homogeneous homomorphism $L_p$ such that 
\begin{align}\label{deriv}
\mathbb{G}_1\ni q\mapsto \delta^{\mathbb{G}_2}_{1/r}\left(f(p)^{-1}\cdot f(p\cdot \delta^{\mathbb{G}_1}_r(q))\right)
\end{align}
tends to $L_p$ uniformly on every compact set $K\subset \mathbb{G}_1$ as $r$ goes to zero. 
\end{prop}
Note that in Proposition \ref{PR_them} the map $L_p$ is uniquely determined. It is called the \emph{Pansu derivative} of $f$ at $p$ and denoted by $Df_p$.

We denote by $\mathcal{C}^1_H(\mathbb{G}_1,\mathbb{G}_2)$ the space of functions $f$ such that \eqref{deriv} holds at every point $p\in\mathbb{G}_1$ and $p\mapsto Df_p$ is continuous for the usual topology. For $\mathbb{G}_1=\R$ this coincides with the definition of $\mathcal{C}^1_H(I,\mathbb{G}_2)$ given earlier. We have the following.
\begin{prop}[Taylor expansion]\label{Taylor}
Let $f\in \mathcal{C}^1_H(\mathbb{G}_1,\mathbb{G}_2)$ where $\mathbb{G}_1$ and $\mathbb{G}_2$ are Carnot groups. Let $K\subset \mathbb{G}_1$ be compact. Then there exists a function $\omega$ from $\R^+$ to $\R^+$ with $\omega(t)=o(t)$ at $0^+$ such that for any $p,q\in K$,
$$d_{\mathbb{G}_2}\left(f(q),f(p)\cdot Df_p(p^{-1}\cdot q)\right)\le \omega(d_{\mathbb{G}_1}(p,q)),$$
where $Df_p$ is the Pansu derivative.
\end{prop}
\begin{proof}
This is a direct consequence the ``mean value inequality''  by Magnani contained in  \cite[Theorem~1.2]{Mag}.
\end{proof}

The above proposition hints at the suitable formulation of the $\mathcal{C}^1$-Whitney condition for Carnot groups. This generalization already appeared in the literature in the paper \cite{VP} by Vodop'yanov and Pupyshev.
\begin{defi}[$\mathcal{C}^1_H$-Whitney condition]

\begin{itemize}
\item
Let $K$ be a compact subset of $\mathbb{G}_1$ and consider 
$f:K\to \mathbb{G}_2$ and a map $L$ which associates with any $p\in K$ 
a homogeneous group homomorphism  $L(p)$. We say that the \emph{$\mathcal{C}^1_H$-Whitney condition holds} for $(f,L)$ on $K$ if $L$ is continuous and there exists a function $\omega$ from $\R^+$ to $\R^+$ with $\omega(t)=o(t)$ at $0^+$ such that for any $p,q\in K$,
\begin{align}\label{Whitney_cond}
d_{\mathbb{G}_2}\left(f(q),f(p)\cdot L(p)(p^{-1}\cdot q)\right)\le \omega(d_{\mathbb{G}_1}(p,q)).
\end{align}
\item
Let $K_0$ be a closed set of $\mathbb{G}_1$, and $f:K_0\to \mathbb{G}_2$, and $L$ such that $K_0\ni p\mapsto L(p)$ is continuous. We say that the \emph{$\mathcal{C}^1_H$-Whitney condition holds} for $(f,L)$ on $K_0$ if for any compact set $K\subset K_0$ it holds 
for the restriction of $(f,L)$ to $K$. 
\end{itemize}
\end{defi}
Of course, according to Proposition \ref{Taylor}, if $f\in \mathcal{C}^1_H(\mathbb{G}_1,\mathbb{G}_2)$, then the restriction of $(f,Df)$ to any closed $K_0$ satisfies the $\mathcal{C}^1_H$-Whitney condition on $K_0$. 

In this paper we focus on the case $\mathbb{G}_1=\R$. The condition on a compact set $K$ reads $r_{K,\eta}\to 0$ as $\eta\to 0$, where
 \begin{align}
 r_{K,\eta}=\sup_{\tau,t\in K,\;0<|\tau-t|<\eta}\frac{d_{\mathbb{G}_2}(f(t),f(\tau)\cdot \exp[(t-\tau)X(\tau)])}{|\tau-t|},
 \end{align}
because for every $\tau\in\R$ one has $[L(\tau)](h)=\exp(hX(\tau))$ for some $X(\tau)\in \mathfrak{G}_H^{\mathbb{G}_2}$ and every $h\in \R$. With a slight abuse of terminology, we say that the $\mathcal{C}^1_H$-Whitney condition holds for $(f,X)$ on $K$. 

 In the classical setting the Whitney condition is equivalent to the existence of a $\mathcal{C}^1$ map $\bar f:\R^{n_1}\to \R^{n_2}$ such that $\bar{f}$ and $D\bar{f}$ have respectively restrictions $f$ and $L$ on $K$. This property is usually known as the \emph{$\mathcal{C}^1$-Whitney extension theorem} or simply \emph{Whitney extension theorem} (as for instance in \cite{EvGa}), even though the original theorem by Whitney is more general and in particular includes higher order extensions \cite{W2,W} and considers the extension $f\to\bar{f}$ as a linear operator. This theorem is of broad use in analysis and is still the subject of dedicated research. 
 See for instance \cite{BS,Fef,FIG} and the references therein.

\begin{defi}\label{def-W-condition}
We say that the pair  $({\mathbb{G}_1},{\mathbb{G}_2})$ \emph{has the $\mathcal{C}^1_H$ extension property} if 
for every $(f,L)$ satisfying the $\mathcal{C}_H^1$-Whitney condition on some closed set $K_0$ there exists $\bar{f}\in \mathcal{C}^1_H({\mathbb{G}_1},{\mathbb{G}_2})$ which extends $f$ on $\mathbb{G}_1$ and such that $D\bar{f}_p=L(p)$ for every $p\in K_0$. 
\end{defi}

We now state the $\mathcal{C}^1_H$-extension theorem that Franchi, Serapioni, and Serra Cassano proved in \cite[Theorem 2.14]{FSS2}. It has been generalised by Vodop'yanov and Pupyshev in \cite{VP,VP0} in a form closer to the original Whitney's result including higher order extensions and the linearity of the operator $f\mapsto \bar{f}$.
\begin{theo}[Franchi, Serapioni, Serra Cassano]
For any Carnot group $\mathbb{G}_1$ and any $d\in\N$, the pair
$(\mathbb{G}_1,\R^d)$ has the $\mathcal{C}^1_H$ extension property.
\end{theo}

The proof proposed by  Franchi, Serapioni, and Serra Cassano is established for Carnot groups of step two only, but is identical for general Carnot groups. It is inspired by the proof in \cite{EvGa}, that corresponds to the special case ${\mathbb{G}_1}=\R^{n_1}$ for $n_1\geq 1$.

Let us mention 
an 
example  from the literature
of non-extension with $\mathbb{G}_1\neq \R$. This remarkable fact was explained to us by A. Kozhevnikov. 
\begin{example}\label{ex:Koz}
If $\mathbb{G}_1$ and $\mathbb{G}_2$ are the ultrarigid Carnot groups of dimension 17 and 16 respectively, presented in \cite{DOW} and analysed in Lemma A.2.1 of \cite{Koz}, 
 one can construct an example $(f,L)$ satisfying the $\mathcal{C}^1_H$-Whitney condition on some compact $K$ without any possible extension $(\bar{f},D\bar{f})$ on $\mathbb{G}_1$. For this, one exploits the rarity of $\mathcal{C}^1_H$ maps of maximal rank in ultrarigid Carnot groups. The definition of ultrarigid from \cite[Definition 3.1]{DOW} is that all quasimorphisms are Carnot similitudes, i.e., a composition of dilations and left-translations. We do not use here directly the definition of ultrarigid groups but just the result stated in Lemma A.2.1 of \cite{Koz} for $\mathbb{G}_1$ and $\mathbb{G}_2$. Concretely, let us set
$$K=\{(p_1,\ldots,p_{17})\in\mathbb{G}_1\mid p_2=\cdots =p_{16}=0,\,p_{1}\in[-1,1],\,p_{17}=p_{1}\}.$$  
Let the map $f$ be constantly equal to $0$ on $K$ and $L$ be the constant  projection $\Lambda:\mathbb{G}_1\ni (q_1,\ldots,q_{17})\mapsto(q_1,\ldots,q_{16})\in \mathbb{G}_2$. Lemma A.2.1 in \cite{Koz} applied at the point $0_{\mathbb{G}_1}$ implies that the only possible extension of $f$ is the projection $L(0)=\Lambda$. But this map vanishes only on $\{p\in\mathbb{G}_1\mid p_{1}=\cdots=p_{16}=0\}$, which does not contain $K$. It remains us to prove that Whitney's condition holds. In fact for two points $p=(x,0_{\R^{15}},x)$ and $q=(y,0_{\R^{15}},y)$ in $K$, we look at the distance from $f(x)=0_{\mathbb{G}_2}$ to 
$$f(p)\cdot L(0)(p^{-1}\cdot q)= L(0)((x,0_{\R^{15}},x)^{-1}\cdot (y,0_{\R^{15}},y))=(y-x,0_{\R^{15}})$$
 on the one side and from $p$ to $q$ on the other side. 
The first one is $|y-x|$, up to a multiplicative constant, and when $|y-x|$ goes to zero the second one is $c|y-x|^{1/3}$ for some constant $c>0$. This proves the $\mathcal{C}^1_H$-Whitney condition for $(f,L)$ on $K$.
\end{example}

In the present paper we provide examples of ordered pairs $(\mathbb{G}_1,\mathbb{G}_2)$ 
with $\mathbb{G}_1=\R$   such that the $\mathcal{C}^1_H$ extension property does or does not hold, depending on the geometry of $\mathbb{G}_2$. We do not address the problem of Whitney extensions for  orders larger than $1$. 
A preliminary step for considering higher-order extensions would be to 
provide a suitable  Taylor expansion for $\mathcal{C}^m_H$-functions from $\R$ to $\mathbb{G}_2$, in the spirit of what recalled for $m=1$ in Proposition~\ref{Taylor}.

Let us conclude the section by assuming that the $\mathcal{C}_H^1$ extension property holds for some ordered pair $(\mathbb{G}_1,\mathbb{G}_2)$ of Carnot groups and by showing how to deduce it for other pairs. We describe here below three such possible implications. 

(1) Let $\mathbb{S}_1$ be a homogeneous subgroup of $\mathbb{G}_1$ that admits a complementary group $\mathbb{K}$ in the sense of \cite[Section 4.1.2]{serra-cassano-notes}: both $\mathbb{S}_1$ and $\mathbb{K}$ are homogeneous Lie groups and the intersection is reduced to $\{0\}$. Assume moreover that $\mathbb{S}_1$ is a Carnot group and $\mathbb{K}$ is normal, so that one can define canonically a projection $\pi:\mathbb{G}_1\to \mathbb{S}_1$ that is a homogeneous homomorphism. Moreover $\pi$ is Lipschitz continuous (see \cite[Proposition 4.13]{serra-cassano-notes}). For the rest of the section, we say that $\mathbb{S}_1$ is an \emph{appropriate Carnot subgroup of $\mathbb{G}_1$}. It can be easily proved that $(\mathbb{S}_1,\mathbb{G}_2)$ has the $\mathcal{C}^1_H$ extension property. In particular, according to  \cite[Example 4.6]{serra-cassano-notes}, for every $k\leq \dim (\mathfrak{G}^{\mathbb{G}_1}_H)$ the vector space $\R^k$ is  an appropriate Carnot subgroup of $\mathbb{G}_1$. Therefore $(\R^k,\mathbb{G}_2)$ has the $\mathcal{C}^1_H$ extension property.

(2) Assume now that $\mathbb{S}_2$ is an appropriate Carnot subgroup of $\mathbb{G}_2$. 
Using the Lipschitz continuity of the projection $\pi:\mathbb{G}_2\to \mathbb{S}_2$,
one easily deduces from the definition of $\mathcal{C}_H^1$-Whitney condition  that $(\mathbb{G}_1,\mathbb{S}_2)$ has the $\mathcal{C}^1_H$ extension property.

(3) Finally assume that $(\mathbb{G}_1,\mathbb{G}'_2)$ has the $\mathcal{C}^1_H$ extension property, where $\mathbb{G}'_2$ is a Carnot group. Then one checks without difficulty that the same is true for $(\mathbb{G}_1,\mathbb{G}_2\times\mathbb{G}'_2)$.

As a consequence of Theorem~\ref{thm-necessary-and-sufficient}, we can use 
these three implications to 
infer pliability statements. Namely, a Carnot group $\mathbb{G}$ is pliable --- i) if $(\mathbb{G}_0,\mathbb{G})$ has the $\mathcal{C}^1_H$ extension property for some Carnot group $\mathbb{G}_0$ of positive dimension, ii) if $\mathbb{G}$ is the appropriate Carnot subgroup of a pliable Carnot group, iii) if $\mathbb{G}$ is the product of two pliable Carnot groups.

\section{Rigidity, Necessary condition for the $\mathcal{C}^1_H$ extension property}\label{sec:necessary}

Let us first adapt to the case of horizontal curves on Carnot groups the notion of rigid curve introduced by Bryant and Hsu in \cite{BH}. 
We will show in the following that the existence of rigid curves in a Carnot group 
$\mathbb{G}$
can be used to identify 
obstructions
to the validity of the $\mathcal{C}^1_H$ extension property for $(\mathbb{R},\mathbb{G})$.

\begin{defi}[Bryant, Hsu]
Let 
$\gamma\in \mathcal{C}^1_H([a,b],\mathbb{G})$.
 We say that \emph{$\gamma$ is rigid} 
 if there exists a  neighborhood $\mathcal{V}$ of $\gamma$  in the space $\mathcal{C}^1_H([a,b],\mathbb{G})$ 
 such that if $\beta\in \mathcal{V}$ and  $\gamma(a)=\beta(a)$, $\gamma(b)=\beta(b)$ then $\beta$ is a reparametrization 
of $\gamma$.

A vector $X\in \mathfrak{G}_H
$ is said to be \emph{rigid} if the curve $[0,1]\ni t\mapsto \exp(t X)$ is rigid.
\end{defi}

A celebrated existence result of rigid curves for general sub-Rieman\-nian manifolds has been obtained by 
Bryant and Hsu in \cite{BH} and further improved in \cite{LiuSussmann} and \cite{agrachev-sarychev-rigidity}. 
Examples of Carnot groups with rigid curves have been illustrated in  \cite{GoleKaridi} and extended in  \cite{HuangYangExtremals}, where it is shown that, for any $N\ge 6$ 
there 
exists a Carnot group of topological dimension $N$ 
having rigid curves. 
Nevertheless, such curves need not be straight. Actually, the construction proposed in \cite{HuangYangExtremals}
produces curves which are necessarily not straight. 

Following \cite{agrachev-sarychev-rigidity} (see also \cite{montgomery-survey}), 
and focusing on rigid straight curves in Carnot groups, 
we can formulate Theorem~\ref{BrHs} below.
In order to state it, let $\pi:T^*\mathbb{G}\to \mathbb{G}$ be  the canonical projection and recall that a curve $p:I\to T^*\mathbb{G}$ is said to be an \emph{abnormal path} 
if $\pi\circ p:I\to\mathbb{G}$ is a horizontal curve, $p(t)\ne 0$ and 
$p(t) X=0$
for every $t\in I$ and $X\in \mathfrak{G}_H$, and, moreover, 
for every $Y\in \mathfrak{G}$ and almost every $t\in I$,  
\begin{equation}
\frac{d}{dt}p(t)Y=p(t)[Z(t),Y],\label{eq-abn}
\end{equation} 
where $Z(t)=\frac{d}{dt} \pi\circ p(t)\in \mathfrak{G}_H$.

\begin{theo}\label{BrHs}
Let $X\in \mathfrak{G}_H$ 
and assume that $p:[0,1]\to T^*\mathbb{G}$ is an abnormal path with $\pi\circ p(t)=\exp(tX)$.

If $t\mapsto \exp(tX)$ is rigid, then $p(t) [V,W]=0$ for every $V,W\in \mathfrak{G}_H$ and every $t\in[0,1]$. 
Moreover, denoting by $Q_{p(t)}$ the quadratic form $Q_{p(t)}(V)=p(t) [V,[X,V]]$ defined on $\{V\in\mathfrak{G}_H\mid V\perp X\}$, we have that $Q_{p(t)}\geq 0$  for every $t\in[0,1]$.

Conversely, if $p(t) [V,W]=0$ for every $V,W\in \mathfrak{G}_H$  and every $t\in[0,1]$ and 
 $Q_{p(t)}> 0$  for  every $t\in[0,1]$ then $t\mapsto \exp(tX)$ is rigid.
\end{theo}

\begin{example}\label{e:Engel}
An example of Carnot structure having rigid straight curves is the standard Engel structure. In this case $s=3$, $\dim\mathfrak{G}_1=2$, $\dim\mathfrak{G}_2=\dim\mathfrak{G}_3=1$ and one can can pick two generators $X,Y$ of the horizontal distribution whose only nontrivial bracket relations are $[X,Y]=W_1$ and $[Y,W_1]=W_2$, where $W_1$ and $W_2$ span $\mathfrak{G}_2$ and $\mathfrak{G}_3$ respectively.

Let us illustrate how the existence of rigid straight curves can be deduced from Theorem~\ref{BrHs} (one could also prove rigidity by direct computations of the same type as those of Example~\ref{e:superEngel} below).

One immediately checks 
that $p$ with $p(t) X=p(t) Y=p(t) W_1=0$ and $p(t) W_2=1$ is an abnormal path 
such that $\pi\circ p(t)=\exp(tX)$. The rigidity of $t\mapsto \exp(tX)$ then follows from 
Theorem~\ref{BrHs}, thanks to the relation  
$Q_p (Y)=1$. 
  
An extension of the previous construction can be used to  exhibit, for every $N\geq 4$,  
a Carnot group of topological dimension $N$ and step $N-1$
having straight rigid curves. It suffices to consider the $N$-dimensional Carnot group with Goursat distribution, that is, the group such that 
 $\dim\mathfrak{G}_1=2$, $\dim\mathfrak{G}_i=1$ for $i=2,\dots,N-1$, and there exist two generators $X,Y$ of $\mathfrak{G}_1$ whose only nontrivial bracket relations are $[X,Y]=W_1$ and $[Y,W_i]=W_{i+1}$ for $i=1,\dots,N-3$, where $\mathfrak{G}_{i+1}=\SPAN(W_i)$ for $i=1,\dots,N-2$. 
\end{example}

The following definition introduces the notion of \emph{pliable} horizontal curve, in contrast to a rigid one.

\begin{defi}\label{d:pliable}
We say that a curve $\gamma\in \mathcal{C}^1_H([a,b],\mathbb{G})$ is \emph{pliable}  if for every 
neighborhood $\mathcal{V}$ of $\gamma$  in $\mathcal{C}^1_H([a,b],\mathbb{G})$
the set
$$\{(\beta(b),\dot{\beta}(b))\mid \beta\in \mathcal{V},\ (\beta,\dot{\beta})(a)=(\gamma,\dot{\gamma})(a)\}$$
is a neighborhood of  $(\gamma(b),\dot{\gamma}(b))$ in $\mathbb{G}\times \mathfrak{G}_H$. 

A vector $X\in \mathfrak{G}_H$ is said to be \emph{pliable} if the curve $[0,1]\ni t\mapsto \exp(tX)$ is pliable.

We say that \emph{$\mathbb{G}$ is pliable} if every  vector $X\in \mathfrak{G}_H$ is pliable.
\end{defi}

By metric equivalence of all 
Carnot--Carath\'eodory distances, it follows that the pliability of a horizontal vector does not depend on the norm $\|\cdot\|_{\mathfrak{G}_H}$ considered on $\mathfrak{G}_H$.

Notice that, by definition of pliability, in every $\mathcal{C}^1_H$ neighborhood of a pliable curve $\gamma:[a,b]\to \mathbb{G}$ there exists a curve $\beta$ with  $\beta(a)=\gamma(a)$,  
$(\beta,\dot{\beta})(b)=(\gamma(b),W)$, and $W\neq \dot{\gamma}(b)$. This shows that pliable curves are not rigid.
 It should be noticed, however, that 
 the converse is not true in general, as will be discussed in Example~\ref{E:notRnotP}.
In this example we show that there exist horizontal straight curves that are neither rigid nor pliable.

\begin{example}\label{E:notRnotP}\label{e:superEngel}
We consider the $6$-dimensional Carnot algebra $\mathfrak{G}$ of step $3$ that is spaned by $X, Y, Z, [X,Z], [Y,Z], [Y,[Y,Z]]$ where $X,\,Y,\,Z$ is a basis of $\mathfrak{G}_1$ and except from permutations all brackets different from the ones above are zero.  

According to \cite[Chapter 4]{Bonfiglioli-et-al} there is a group structure on $\R^6$  with coordinates $(x,y,z,z_1,z_2,z_3)$ isomorphic to the corresponding Carnot group $\mathbb{G}$ such that the vectors of $\mathfrak{G}_1$ are the left-invariant vector fields 
$$X=\partial_x,\quad Y=\partial_y, \quad Z=\partial_z+x\partial_{z_1}+y\partial_{z_2}+y^2\partial_{z_3}.$$

Consider the straight curve $[0,1]\ni t\mapsto \gamma(t)=\exp(t Z)\in \mathbb{G}$.
First notice that \emph{$\gamma$ is not pliable}, since for all horizontal curves in a small enough $\mathcal{C}^1$ neighbourhood of $\gamma$ the component of the derivative along $Z$ is positive, which implies that the coordinate $z_3$ is nondecreasing. No endpoint of a horizontal curve starting from $0_\mathbb{G}$ and belonging to  a small enough $\mathcal{C}^1$ neighbourhood of $\gamma$ can have negative $z_3$ component. 

Let us now show that \emph{$\gamma$ is not rigid} either. 
Consider the solution $\beta$ of
$$\dot\beta(t)=Z(\beta(t))+u(t) X(\beta(t)),\qquad \beta(0)=0_\mathbb{G}.$$ 
Notice that the $y$ component of $\beta$ is identically equal to zero. 
As a consequence, the same is true for the components $z_2$ and $z_3$, while the $x$, $z$ and $z_1$  components of $\beta(t)$ are, respectively, $\int_0^t u(\tau)d\tau$, $t$, and $\int_0^t\int_0^\tau u(\theta)d\theta d\tau$. 
In order to disprove the rigidity, it is then sufficient to 
take a nontrivial continuous $u:[0,1]\to \R$ such that $\int_0^1 u(\tau)d\tau=0=\int_0^1\int_0^\tau u(\theta)d\theta d\tau$. 
\end{example}

Let us list some useful manipulations which transform horizontal curves into horizontal curves. Let $\gamma$ be a horizontal curve  defined on $[0,1]$ and such that $\gamma(0)=0_\mathbb{G}$. 
\begin{enumerate}[label=(T\arabic*)]
\item \label{T1} For every $\lambda>0$, the curve $t\in [0,\lambda]\mapsto\delta_\lambda\circ\gamma(\lambda^{-1}t)$ is horizontal and its velocity at time $t$ is $\dot\gamma(\lambda^{-1}t)$. 
\item \label{T2} For every $\lambda<0$, the curve $t\in [0,|\lambda|]\mapsto\delta_{\lambda}\circ\gamma(|\lambda|^{-1}t)$ is horizontal and its velocity at time $t$ is $-\dot\gamma(|\lambda|^{-1}t)$.
\item \label{T3} The curve $\bar{\gamma}$ defined by $\bar{\gamma}(t)=\gamma(1)^{-1}\cdot \gamma(1-t)$ is horizontal. It starts in $0_\mathbb{G}$ and finishes in $\gamma^{-1}(1)$. Its velocity at time $t$ is $-\dot{\gamma}(1-t)$.
\item \label{T4} If one composes the (commuting) transformations \ref{T2} with $\lambda=-1$ and \ref{T3}, one obtains a curve with derivative $\dot\gamma(1-t)$ at time $t$.
\item \label{T5} It is possible to define the concatenation of two curves $\gamma_1:[0,t_1]\to \mathbb{G}$ and  $\gamma_2:[0,t_2]\to \mathbb{G}$ 
both starting from $0_\mathbb{G}$ as follows: the concatenated curve $\tilde\gamma:[0,t_1+t_2]\to \mathbb{G}$ satisfies $\tilde{\gamma}(0)=0_\mathbb{G}$, has the same velocity as $\gamma_1$ on $[0,t_1]$ and the velocity of $\gamma_2(\cdot\,-t_1)$ on $[t_1,t_1+t_2]$. We have $\tilde{\gamma}(t_1+t_2)=\gamma_1(t_1)\cdot \gamma(t_2)$ as a consequence of the invariance of the the Lie algebra for the left-translation.
\end{enumerate}

A consequence of \ref{T1} and \ref{T2} is that  $X\in \mathfrak{G}_H\setminus\{0\}$ is rigid if and only if $\lambda X$ is rigid for every $\lambda\in \R\setminus\{0\}$. 
Similarly, $X\in \mathfrak{G}_H$ is pliable if and only if $\lambda X$ is pliable for every $\lambda\in \R\setminus\{0\}$.

Proposition~\ref{p:pliableiffweaklypliable} below gives a characterization 
of pliable horizontal vectors in terms of a condition which is apriori easier to check than the one appearing in Definition~\ref{d:pliable}.  
Before proving the proposition, let us give a technical lemma.
From now on, we write $\mathcal{B}_\mathbb{G}(x,r)$ to denote the 
ball of center $x$ and radius $r$ in $\mathbb{G}$ for the distance $d_{\mathbb{G}}$ and, similarly, $\mathcal{B}_{\mathfrak{G}_H}(x,r)$ to denote the 
ball of center $x$ and radius $r$ in $\mathfrak{G}_H$ for the norm $\|\cdot\|_{\mathfrak{G}_H}$.

\begin{lem}\label{brick}
For any $x\in \mathbb{G}$ and $0<r<R$, there exists $\eps>0$ 
such that if $y,z\in \mathbb{G}$ and $\rho\geq 0$ satisfy $d_\mathbb{G}(y,0_\mathbb{G}),d_\mathbb{G}(z,0_\mathbb{G}),\rho \leq \eps$, 
then
$$ \mathcal{B}_\mathbb{G}(x,r)\subset y\cdot\delta_{1-\rho}(\mathcal{B}_\mathbb{G}(x,R))\cdot z.$$
\end{lem}
\begin{proof}
Assume, by contradiction, that for every $n\in \N$ there exist $x_n\in \mathcal{B}_\mathbb{G}(x,r)$, $y_n,z_n\in \mathcal{B}_\mathbb{G}(0_\mathbb{G},1/n)$ and $\rho_n\in [0,1/n]$ such that 
$$x_n\not\in y_n\cdot\delta_{1-\rho_n}(\mathcal{B}_\mathbb{G}(x,R))\cdot z_n.$$
Equivalently, 
$$ \delta_{(1-\rho_n)^{-1}}(y_n^{-1}\cdot x_n\cdot z_n^{-1})\not\in \mathcal{B}_\mathbb{G}(x,R).$$
However, $\limsup_{n\to \infty}d_\mathbb{G}(x,\delta_{(1-\rho_n)^{-1}}(y_n^{-1}\cdot x_n\cdot z_n^{-1}))\leq r$, leading to a contradiction. 
\end{proof}

\begin{prop}\label{p:pliableiffweaklypliable}\label{p:wp-vwp}
A vector $V\in \mathfrak{G}_H$ is pliable 
if and only if
for every 
neighborhood $\mathcal{V}$ of the curve $[0,1]\ni t\mapsto \exp(tV)$  in the space 
$\mathcal{C}^1_H([0,1],\mathbb{G})$, the set
$$\{\beta(1)\mid \beta\in \mathcal{V},\ (\beta,\dot{\beta})(0)=(0_\mathbb{G},V)\}$$
is a  neighborhood of $
\exp(V)$.
\end{prop}
\begin{proof}
Let $\mathcal{F}:\mathcal{C}^1_H([0,1],\mathbb{G})\ni \beta\mapsto(\beta,\dot{\beta})(1)\in \mathbb{G}\times \mathfrak{G}_H$ and denote by $\pi: \mathbb{G}\times \mathfrak{G}_H\to \mathbb{G}$ the canonical projection.

One direction of the equivalence being trivial, let us take $\eps>0$ and assume that $\pi\circ \mathcal{F}(\mathcal{U}_\eps)$ is a neighbourhood of $\exp(V)$ in $\mathbb{G}$, where
$$\mathcal{U}_\eps=\{\beta\in \mathcal{C}^1_H([0,1],\mathbb{G})\mid (\beta,\dot{\beta})(0)=(0_\mathbb{G},V),\;\|\dot\beta-V\|_{\infty,\mathfrak{G}_H}<\eps\}.$$ 
 We should prove that $\mathcal{F}(\mathcal{U}_\eps)$ is a neighborhood of $(\exp(V),V)$ in $\mathbb{G}\times \mathfrak{G}_H$. 

\emph{Step 1:} As an intermediate step, we first prove that 
there exists 
$\eta>0$ such that $\mathcal{B}_\mathbb{G}(\exp(V),\eta)\times\{V\}$ is contained 
in 
$\mathcal{F}(\mathcal{U}_\eps)$.

Let $\rho$ be a real parameter in $(0, 1)$. 
Using the transformations among horizontal curves 
described earlier in this section, let us define a map $T_\rho:\mathcal{U}_\eps\to \mathcal{C}^1_H([0,1],\mathbb{G})$
associating with a curve $\gamma\in \mathcal{U}_\eps$ 
the concatenation (transformation \ref{T5})
of $\gamma_1:t\mapsto \delta_\rho\circ\gamma(\rho^{-1} t)$ on $[0,\rho]$ obtained by transformation \ref{T1} and a curve $\gamma_2$ 
defined as follows. 
Consider $\gamma_{2,1}:[0,1-\rho]\ni t\mapsto \delta_{1-\rho}\circ\gamma((1-\rho)^{-1}t)$ (again \ref{T1}). The curve $\gamma_2$ is defined from $\gamma_{2,1}$ by 
$$\gamma_2(t)=\gamma_1(\rho)\cdot \left(\gamma_{2,1}(1-\rho)^{-1}\cdot \delta_{-1}\circ\gamma_{2,1}((1-\rho)-t)\right)$$
(see transformation \ref{T4}). The derivative of $T_\rho(\gamma)$ at time $t\in[0, \rho)$ is $\dot{\gamma}(\rho^{-1}t)$. Its derivative at time $\rho+t$ is $\dot\gamma(1-(1-\rho)^{-1}t)$ for $t\in(0,1-\rho]$. 
Hence $T_\rho(\gamma)$ is continuous and has derivative $\dot{\gamma}(1)$ at limit times $\rho^-$ and $\rho^+$, i.e., is a well-defined map from $\mathcal{U}_\eps$ into $\mathcal{C}^1_H([0,1],\mathbb{G})$. 
Moreover,   $T_\rho(\gamma)$
has the same derivative $V=\dot{\gamma}(0)$ at times $0$ and $1$
and its derivative at any time in $[0,1]$ is in the set of the derivatives of $\gamma$. In particular, 
$T_\rho(\mathcal{U}_\eps)\subset \mathcal{U}_\eps$.

Notice  now that, by construction, the endpoint $T_\rho(\gamma)(1)$ of the curve $T_\rho(\gamma)$ 
is a function of $\gamma(1)$ and $\rho$ only. It is actually equal to 
$$F_\rho(x)=\delta_\rho(x)\cdot \delta_{\rho-1}(x)^{-1},$$
where $x=\gamma(1)$ (see \ref{T1} and \ref{T4}) . Let $x_0=\exp(V)$ and $\gamma_0:t\mapsto \exp(tV)$. We have $F_\rho(x_0)=x_0$ because $T_\rho(\gamma_0)=\gamma_0$, both curves having derivative constantly equal to $V$.
We prove now that for $\rho$ close enough to $1$, 
the differential of $F_\rho$ at $x_0$ is invertible. 
Let us use the coordinate identification of $\mathbb{G}$ with $\R^N$.
For every $y\in \mathbb{G}$, the limits of $\delta_\rho(y)$ and $\delta_{1-\rho}(y)$ 
as $\rho$ tends to $1$ are $y$ and $0_\mathbb{G}$ respectively,
while $D\delta_\rho(y)$ and $D\delta_{\rho-1}(y)$ converge to $\mathrm{Id}$ and $0$ respectively. One can check (see, e.g., \cite[Proposition 2.2.22]{Bonfiglioli-et-al}) that the inverse function has derivative $-\mathrm{Id}$ at $0_\mathbb{G}$. Finally the left and right translations are global diffeomorphisms. Collecting these informations and applying the chain rule, we get that $DF_\rho(x_0)$ tends to an invertible operator as $\rho$ goes to $1$. Hence for $\rho$ great enough, $F_\rho(x_0)$ is a local diffeomorphism.

We know by assumption on $V$  that, for any $\eps>0$, the endpoints of the curves of $\mathcal{U}_\eps$ 
form a neighborhood of $x_0$. 
We have shown  that this is also the case if we replace $\mathcal{U}_\eps$ by $T_\rho(\mathcal{U}_\eps)$,  for $\rho$ close to $1$. The curves of $T_\rho(\mathcal{U}_\eps)$ are in $\mathcal{U}_\eps$ and  have, moreover, derivative $V$  at time $1$.
He have thus proved that
for every  $\eps>0$ there exists 
$\eta>0$ such that $\mathcal{B}_\mathbb{G}(x_0,\eta)\times\{V\}$ is contained 
in 
$\mathcal{F}(\mathcal{U}_\eps)$.

\emph{Step 2:} Let us now prove that $\mathcal{F}(\mathcal{U}_\eps)$ is a neighborhood of $(x_0,V)$ in $\mathbb{G}\times \mathfrak{G}_H$. 

Let $\beta$ be a curve in $\mathcal{U}_\eps$ with $\dot\beta(1)=V$ and consider for every $W\in \mathcal{B}_{\mathfrak{G}_H}(V,\eps)$ and every $\rho\in(0,1)$ the curve $\alpha_{\rho,W}$ defined as follows: $\alpha_{\rho,W}=\delta_{1-\rho}\circ \beta((1-\rho)^{-1} t)$ on $[0,1-\rho]$ (transformation \ref{T1}) and $\dot{\alpha}_{\rho,W}$ is the linear interpolation between $V$ and $W$ on $[1-\rho,1]$. 
Notice that $\alpha_{\rho,W}$ is in $\mathcal{U}_\eps$. 

Let $u\in \mathbb{G}$ be the endpoint at time $\rho$ of the curve in $\mathbb{G}$ starting at $0_\mathbb{G}$
whose derivative is the linear interpolation between $V$ and $W$ on $[0,\rho]$.
Then 
$(\alpha_{\rho,W},\dot\alpha_{\rho,W})(1)=(\delta_{1-\rho}(\beta(1))\cdot u,W)$ and $u$ depends only on $V$, $\rho$ and $W$, and not on the curve $\beta$. Moreover, $u$ tends to $0_\mathbb{G}$ as $\rho$ goes to $1$, uniformly with respect to $W\in \mathcal{B}_{\mathfrak{G}_H}(V,\eps)$. 
Lemma~\ref{brick} implies that for $\rho$ sufficiently close to $1$, for every $W\in \mathcal{B}_{\mathfrak{G}_H}(V,\eps)$, it holds $\delta_{1-\rho}(\mathcal{B}_\mathbb{G}(x_0,\eta))\cdot u\supset\mathcal{B}_\mathbb{G}(x_0,\eta/2)$. 
We proved that $\mathcal{B}_\mathbb{G}(x_0,\eta/2)\times \mathcal{B}_{\mathfrak{G}_H}(V,\eps)\subset \mathcal{F}(\mathcal{U}_\eps)$, concluding the proof of the proposition. 
\end{proof}

 The main result of this section is the following theorem, which 
constitutes  the necessity part of the characterization of 
 $\mathcal{C}^1_H$ extendability
 stated in Theorem~\ref{thm-necessary-and-sufficient}.

\begin{theo}\label{thm-necessary}
Let $\mathbb{G}$ be a Carnot group. 
If  $(\R,\mathbb{G})$ has the $\mathcal{C}^1_H$ extension property, then $\mathbb{G}$ is pliable.
\end{theo}

\begin{proof}
Suppose, by contradiction, that there exists $V\in\mathfrak{G}_H$ which is not pliable. We are going to prove that $(\R,\mathbb{G})$ has not the $\mathcal{C}^1_H$ extension property.

Let $\gamma(t)=\exp(tV)$ for $t\in [0,1]$.
Since $V$ 
is not 
pliable,  
it follows from Proposition~\ref{p:pliableiffweaklypliable} that 
there exist a 
neighborhood $\mathcal V$ of $\gamma$ 
in the space $\mathcal{C}^1_H([0,1],\mathbb{G})$ 
and a sequence $(x_n)_{n\ge 1}$ converging to $0_\mathbb{G}$ 
such that for every $n\ge 1$ 
no curve
$\beta$ 
in  $\mathcal {V}$ satisfies $(\beta(0),\dot\beta(0))=(0_\mathbb{G},V)$ and $\beta(1)= \gamma(1)\cdot x_n$. 
In particular, there exists a neighborhood $\Omega$ of $V$ in $\mathfrak{G}_H$ such that for every 
$\beta\in \mathcal{C}^1_H([0,1],\mathbb{G})$ with $(\beta(0),\dot\beta(0))=(0_\mathbb{G},V)$ and 
$$\dot\beta(t)\in \Omega,\quad \forall t\in[0,1],$$
we have $(\beta(1),\dot\beta(1))\ne (  \gamma(1)\cdot x_n, V)$ for every $n\in\N$. Since $\lim_{n\to\infty}x_n=0_\mathbb{G}$, we can assume without loss of generality 
that, 
for every $n\ge 1$,
\begin{equation}\label{dinfty1}
\max\{d(\delta_\rho(x_n)\cdot\exp(t V),\exp(t V))\mid \rho\in[0,1],\;t\in[-1,1]\}\leq 2^{-n}.
\end{equation}

By homogeneity and left-invariance, we deduce that for every $y\in \mathbb{G}$ and every $\rho>0$, 
 for every 
$\beta\in \mathcal{C}^1([0,\rho],\mathbb{G})$ with $(\beta(0),\dot\beta(0))=(y,V)$ and 
$$\dot\beta(t)\in \Omega,\quad \forall t\in[0,\rho],$$
we have $(\beta(\rho),\dot\beta(\rho))\ne (y\cdot \gamma(\rho)\cdot \delta_\rho( 
x_n), V)$ for every $n\in\N$.

Define $\rho_n=\frac{1}n-\frac{1}{n+1}=\frac{1}{n(n+1)}$ and $\tilde x_n=\delta_{\rho_n}(x_n)$ for every $n\in \N$. It follows from \eqref{dinfty1} that
\begin{equation}\label{dinfty}
\max\{d(\tilde x_n\cdot\exp(t V),\exp(t V))\mid t\in[-1,1]\}\leq 2^{-n},
\qquad \forall n\ge 1.
\end{equation}
We introduce the sequence defined recursively by $y_0=0_\mathbb{G}$ and 
\begin{equation}\label{yn}
y_{n+1}=y_n\cdot \gamma(\rho_n)\cdot \tilde x_n.
\end{equation}
Notice that $(y_n)_{n\geq 1}$ is a Cauchy sequence and denote by $y_\infty$ its limit as $n\to \infty$.

By construction, for every  $n\in\N$ and every
$\beta\in \mathcal{C}^1_H([0,\rho_n],\mathbb{G})$ with $(\beta(0),\dot\beta(0))=(y_n,V)
$ and 
$\dot\beta(t)\in \Omega$ for all $t\in[0,\rho_n]$,
we have $(\beta(\rho_n),\dot\beta(\rho_n))\ne (y_{n+1}, V)
$.
The proof that the $(\R,\mathbb{G})$ has not the $\mathcal{C}^1_H$ extension property is then concluded if we show that the $\mathcal{C}^1_H$-Whitney condition holds for $(f,X)$ on $K$, where 
$$K=\left(\cup_{n=1}^{\infty}\left\{1-\frac1n\right\}\right)\bigcup\{1\},$$
and $f:K\to \mathbb{G}$ and $X:K\to \mathfrak{G}_H$ are defined by 
$$f(1-n^{-1})=y_n,\quad 
X(1-n^{-1})=V,\qquad n\in\N^*\cup\{\infty\}.$$

For $i,j\in \N^*\cup\{\infty\}$, let 
\begin{align*}
D(i,j)&=d_\mathbb{G}(f(1-i^{-1}),  f(1-j^{-1})\cdot \exp[(j^{-1}-i^{-1})X(1-j^{-1})] )\\
&=d_\mathbb{G}(y_i, y_j\cdot \exp[(j^{-1}-i^{-1})V]).
\end{align*}

We have to prove that
$$D(i,j)=o(j^{-1}-i^{-1})$$
as $i,j\to \infty$, that is, for every $\eps>0$, there exists $i_\eps\in \N^*$ such that $D(i,j)<\eps|j^{-1}-i^{-1}|$ for $i,j\in\N^*\cup\{\infty\}$ with $i,j>i_\eps$. 

By triangular inequality we have
\begin{align*}
D(i,j)&\leq 
\sum_{k=\min(i,j)}^{\max(i,j)-1}  d_\mathbb{G}(y_{k+1}\cdot\exp[((k+1)^{-1}-i^{-1})V],y_k\cdot\exp[(k^{-1}-i^{-1})V]).
\end{align*}

Notice that 
\begin{align*}
d_\mathbb{G}&(y_{k+1}\cdot\exp[((k+1)^{-1}-i^{-1})V],y_k\cdot\exp[(k^{-1}-i^{-1})V])\\
&= d_\mathbb{G}(y_{k+1}\cdot\exp[((k+1)^{-1}-i^{-1})V],[y_k\cdot\gamma(\rho_k)]\cdot\exp[((k+1)^{-1}-i^{-1})V])\\
&= d_\mathbb{G}(\tilde x_{k}\cdot\exp[((k+1)^{-1}-i^{-1})V],\exp[((k+1)^{-1}-i^{-1})V]),
\end{align*}
where the last equality 
follows from \eqref{yn} and the invariance of $d_\mathbb{G}$ by left-multiplication.
Thanks to \eqref{dinfty}, one then concludes that  
$$d_\mathbb{G}(y_{k+1}\cdot\exp[((k+1)^{-1}-i^{-1})V],y_k\cdot\exp[(k^{-1}-i^{-1})V])\le 2^{-k}.$$

Hence, $D(i,j)\leq \sum_{k=\min(i,j)}^{\max(i,j)-1} 2^{-k}=o(j^{-1}-i^{-1})$ 
and this concludes the proof 
of Theorem~\ref{thm-necessary}.
\end{proof}

\section{Sufficient condition for the $\mathcal{C}^1_H$ extension property}\label{sec:sufficient}

We have seen in the previous section 
that, differently from the classical case, for a general Carnot group $\mathbb{G}$
the suitable Whitney condition for $(f,X)$ on $K$ is not sufficient for the existence of an extension $(f,\dot f)$ of $(f,X)$ on $\R$. 
More precisely, it follows from Theorem~\ref{thm-necessary} that if $\mathbb{G}$ has horizontal vectors which are not pliable, then there exist  
triples $(K,f,X)$ such that the $\mathcal{C}^1_H$-Whitney condition holds for $(f,X)$ on $K$ but there is not a $\mathcal{C}^1_H$-extension of $(f,X)$. In this next section we 
prove the converse to the result above, showing that the $\mathcal{C}_H^1$ extension property holds when 
all horizontal vectors are pliable, i.e., when $\mathbb{G}$ is pliable. 

We start by introducing the notion of \emph{locally uniformly pliable} horizontal vector. 

\begin{defi}\label{def:up}
A horizontal vector $X$ is called \emph{locally uniformly pliable} if there exists a neighborhood $\mathcal{U}$ of $X$ in $\mathfrak{G}_H$ such that for every $\eps>0$, there exists $\eta>0$ so that for every $W\in \mathcal{U}$
\begin{align*}\{(\gamma,\dot{\gamma})(1)\mid\gamma\in \mathcal{C}^1_H([0,1],\mathbb{G}),\;&(\gamma,\dot{\gamma})(0)=(0_\mathbb{G},W),\;  \|\dot\gamma - W\|_{\infty,\mathfrak{G}_H}\leq \eps\}\\
&\supset \mathcal{B}_\mathbb{G}(\exp(W),\eta)\times\mathcal{B}_{\mathfrak{G}_H}(W,\eta).
\end{align*}
\end{defi}

\begin{remark}\label{r:up-homogeneous}
As it happens for pliability, 
if $X$ is locally uniformly pliable then, for every $\lambda\in\R\setminus\{0\}$, $\lambda X$ is locally uniformly pliable.
\end{remark}

We are going to see in the following (Remark~\ref{zero-not-unif})
that pliability and local uniform pliability are not equivalent properties. 
The following proposition, however, establishes the equivalence between 
pliability and local uniform pliability of \emph{all} horizontal vectors.

\begin{prop}\label{p=>up}
If $\mathbb{G}$ is pliable, then all horizontal vectors are locally uniformly pliable. 
\end{prop}
\begin{proof}
Assume that $\mathbb{G}$ is pliable. 
For every $V\in \mathfrak{G}_H$ and $\eps>0$
denote by $\eta(V,\eps)$ a positive constant such that
\begin{align*}\{(\gamma,\dot{\gamma})(1)\mid\gamma\in \mathcal{C}^1_H([0,1],\mathbb{G}),\;&(\gamma,\dot{\gamma})(0)=(0_\mathbb{G},V),\;  \|\dot\gamma - V\|_{\infty,\mathfrak{G}_H}\leq \eps\}\\
&\supset \mathcal{B}_\mathbb{G}(\exp(V),\eta(V,\eps))\times\mathcal{B}_{\mathfrak{G}_H}(V,\eta(V,\eps)).
\end{align*}

We are going to show that there exists $\nu(V,\eps)>0$ such that 
for every $W\in \mathcal{B}_{\mathfrak{G}_H}(V,\nu(V,\eps))$
\begin{align}\{(\gamma,\dot{\gamma})(1)\mid\gamma\in &\mathcal{C}^1_H([0,1],\mathbb{G}),\;(\gamma,\dot{\gamma})(0)=(0_\mathbb{G},W),\;  \|\dot\gamma - W\|_{\infty,\mathfrak{G}_H}\leq \eps\}\nonumber\\
&\supset \mathcal{B}_\mathbb{G}\left(\exp(W),\frac{\eta\left(V,\frac\eps2\right)}{4}\right)\times\mathcal{B}_{\mathfrak{G}_H}\left(W,\frac{\eta\left(V,\frac\eps2\right)}{4}\right).\label{uniform-eta-delta}
\end{align}

The proof of the local uniform pliability of any horizontal vector $X$ is then concluded by simple compactness arguments (taking any compact neighborhood $\mathcal{U}$ of $X$, using the notation of Definition~\ref{def:up}).

First fix $\bar \nu(V,\eps)>0$ in such a way that 
$$\exp(W)\in \mathcal{B}_\mathbb{G}(\exp(V),\eta(V,\eps/2)/4)$$
for every $W\in \mathcal{B}_{\mathfrak{G}_H}(V,\bar \nu(V,\eps))$.

For every $W\in \mathfrak{G}_H$, every $\rho\in (0,1)$, and every curve  $\gamma\in \mathcal{C}^1_H([0,1],\mathbb{G})$ such that $(\gamma,\dot{\gamma})(0)=(0_\mathbb{G},V)$, define
$\gamma_{W,\rho}\in \mathcal{C}^1_H([0,1],\mathbb{G})$ as follows: $\gamma_{W,\rho}(0)=0_\mathbb{G}$, $\dot\gamma_{W,\rho}(t)=(t/\rho)V+((\rho-t)/\rho)W$ for $t\in [0,\rho]$, 
$\dot\gamma_{W,\rho}(\rho+(1-\rho)t)=\dot\gamma(t)$ for $t\in [0,1]$. 
In particular
$$\gamma_{W,\rho}(1)=\gamma_{W,\rho}(\rho)\cdot \delta_{1-\rho}(\gamma(1)),\quad \dot \gamma_{W,\rho}(1)=\dot\gamma(1),$$
and
$$
\|\dot\gamma_{W,\rho} - W\|_{\infty,\mathfrak{G}_H}\leq \|\dot\gamma - V\|_{\infty,\mathfrak{G}_H}+ 
\|W-V\|_{\mathfrak{G}_H}.$$
If $\|V-W\|_{\mathfrak{G}_H}\leq \eps/2$, 
we then have 
$$
\|\dot\gamma_{W,\rho} - W\|_{\infty,\mathfrak{G}_H}\leq \eps\quad \forall\gamma\mbox{ such that }\|\dot\gamma - V\|_{\infty,\mathfrak{G}_H}\leq \frac\eps2.
$$
Since $\gamma_{W,\rho}(\rho)$ depends on $V,W$, and $\rho$, but not on $\gamma$, 
 we conclude that, for every $W\in \mathcal{B}_{\mathfrak{G}_H}(V,\eps/2
 )$,
\begin{align*}\{(\beta,&\dot{\beta})(1)\mid\beta\in \mathcal{C}^1_H([0,1],\mathbb{G}),\;(\beta,\dot{\beta})(0)=(0_\mathbb{G},W),\;  \|\dot\beta - W\|_{\infty,\mathfrak{G}_H}\leq \eps\}\nonumber\\
&\supset \left(
\gamma_{W,\rho}(\rho)\cdot \delta_{1-\rho}\left(\mathcal{B}_\mathbb{G}\left(\exp(V),\eta\left(V,\frac \eps2\right)\right)\right)\right)\times\mathcal{B}_{\mathfrak{G}_H}\left(V,\eta\left(V,\frac\eps2\right)\right).
\end{align*}

Notice that 
$d_\mathbb{G}(0_\mathbb{G},\gamma_{W,\rho}(\rho))\leq \rho \max(\|V\|_{\mathfrak{G}_H},\|W\|_{\mathfrak{G}_H})$. 
Thanks to Lemma~\ref{brick}, for $\rho$ sufficiently small, 
$$\gamma_{W,\rho}(\rho)\cdot \delta_{1-\rho}\left(\mathcal{B}_\mathbb{G}\left(\exp(V),\eta\left(V,\frac \eps2\right)\right)\right)\supset 
\mathcal{B}_\mathbb{G}\left(\exp(V),\frac{\eta\left(V,\frac\eps2\right)}{2}\right)
.$$
Now, 
$$\mathcal{B}_\mathbb{G}\left(\exp(V),\frac{\eta\left(V,\frac\eps2\right)}{2}\right)\supset 
\mathcal{B}_\mathbb{G}\left(\exp(W),\frac{\eta\left(V,\frac\eps2\right)}{4}\right)
$$
 whenever $W\in \mathcal{B}_{\mathfrak{G}_H}(V,\bar\nu(V,\eps))$.
 
Similarly, 
$$\mathcal{B}_{\mathfrak{G}_H}\left(V,\eta\left(V,\frac\eps2\right)\right)\supset 
 \mathcal{B}_{\mathfrak{G}_H}\left(W,\frac{\eta\left(V,\frac\eps2\right)}{4}\right),$$
 provided that $\|V-W\|_{\mathfrak{G}_H}\leq 3\eta(V,\eps/2)/4$.
 The proof of 
\eqref{uniform-eta-delta} is concluded by taking 
$\nu(V,\eps)=\min(\bar\nu(V,\eps),\eps/2,3\eta(V,\eps/2)/4)$. 
\end{proof}

We are now ready to prove the converse of Theorem~\ref{thm-necessary}, concluding the proof of 
Theorem~\ref{thm-necessary-and-sufficient}. 
\begin{theo}\label{thm-sufficient}
Let $\mathbb{G}$ be a pliable Carnot group. Then $(\R,\mathbb{G})$ has the $\mathcal{C}^1_H$ extension property.
\end{theo}
\begin{proof}
By Proposition~\ref{p=>up}, we can assume that  all vectors in $\mathfrak{G}_H$ are locally uniformly pliable. Note moreover that it is enough to prove the extension for maps defined on compact sets $K$. The generalisation to closed sets $K_0$ is immediate because the source Carnot group is $\R$. Let $(f,X)$ satisfy the $\mathcal{C}^1_H$-Whitney condition on $K$ where $K$ is compact. We have to define $\bar{f}$ on the complementary (open) set $\R\setminus K$, which is the countable and disjoint union of open intervals. For the unbounded components of $\R\setminus K$, we simply define $\bar{f}$ as the curve with constant speed $X(i)$ or $X(j)$ where $i=\min(K)$ and $j=\max(K)$. For the finite components $(a,b)$ we proceed as follows. 
We consider 
$y=\delta_{1/(b-a)}(f(a)^{-1}\cdot f(b))$. We let $\eps$ be the smallest number such that
$$\{(\gamma,\dot{\gamma})(1)\mid \gamma\in \mathcal{C}^1_H([0,1],\mathbb{G}),\;(\gamma,\dot{\gamma})(0)=(0_\mathbb{G},X(a)), \|\dot \gamma - X(a)\|_{\infty,\mathfrak{G}_H}\leq \eps'\}$$
contains $(y,X(b))$ for every $\eps'>\eps$. We consider an extension $\bar{f}\in \mathcal{C}^1_H$ of $f$ on $[a,b]$ such that $\dot{\bar{f}}(a)=X(a)$, $\dot{\bar{f}}(b)=X(b)$,  and $\|\dot{\bar{f}}- X(a)\|_{\infty,\mathfrak{G}_H}\leq 2\eps$.
By definition of the $\mathcal{C}^1_H$-Whitney condition, there exists a function $\omega:\R^+\to \R^+$ tending to $0$ at $0$ such that
$R(a,b)=d_\mathbb{G}(f(b),f(a)\cdot \exp[(b-a)X(a)])$
is smaller than $\omega(b-a)(b-a)$ and $\|X(b)-X(a)\|_{\mathfrak{G}_H}\leq \omega(b-a)$. 
Since $R(a,b)$ is equal to $(b-a)d_\mathbb{G}(\exp(X(a)),y)$, we can conclude that 
\begin{equation}\label{xyz}
d_\mathbb{G}(\exp(X(a)),y)\leq \omega(b-a).
\end{equation} 
Using the corresponding estimates for $R(b,a)$, we deduce that 
\begin{equation}\label{yzx}
d_\mathbb{G}(\exp(-X(b)),y^{-1})\leq \omega(b-a).
\end{equation} 

By construction $\bar{f}$ extends $f$ and $\dot{\bar{f}}=X$ on the interior of $K$.
We prove now that $\bar{f}$ is $\mathcal{C}^1_H$ and that $\dot {\bar{f}}=X$ on the boundary $\partial K$ of $K$. It is clear that $\bar{f}$ is $\mathcal{C}^1_H$ on $\R\setminus \partial K$. In order to conclude the proof we are left to pick $x\in \partial K$, let $x_n$ tend to $x$, and we must show that $\bar{f}(x_n)$ and $\dot{\bar{f}}(x_n)$ tend to $f(x)$ and $X(x)$ respectively. As $f$ and $X$ are continuous on $K$, we can assume without loss of generality that each $x_n$ is in $\R\setminus K$. Assume 
for now
that $x_n<x$ for every $n$. The connected component $(a_n,b_n)$ of $\R\setminus K$ containing $x_n$ is either constant for $n$ large (in this case $x=b_n$)  or its length goes to zero as $n\to \infty$. 
In the  first case we simply notice that $\bar{f}|_{[a_n,b_n]}$ is $\mathcal{C}^1$ by construction. 
In the second case we can assume that  $a_n<x_n<b_n$ and $b_n-a_n$ goes to zero. As $f$ and $X$ are continuous, $f(a_n)$ and $X(a_n)$ converge to $f(x)$ and $X(x)$ respectively.
Inequality  \eqref{xyz} guarantees that 
$d_\mathbb{G}(\exp(X(a_n)),\delta_{1/(b_n-a_n)}[f(a_n)^{-1}\cdot f(b_n)])\leq \omega(b_n-a_n)\to 0$ as $n\to \infty$ and the local uniform pliability of $X(x)$ implies that  
 $\|\dot{\bar{f}}|_{[a_n,b_n]}-X(a_n)\|_{\infty,\mathfrak{G}_H}$ goes to zero as $n\to \infty$. It follows that $\|\dot {\bar{f}}(x_n)-X(a_n)\|_{\mathfrak{G}_H}
 $ and $d_\mathbb{G}(\bar{f}(x_n),f(a_n))$ go to zero, proving that $\bar{f}(x_n)$ and $\dot {\bar{f}}(x_n)$ tend to $f(x)$ and $X(x)$ respectively.
 The situation where $x_n>x$ for infinitely many $n$ can be handled similarly replacing \eqref{xyz} by  \eqref{yzx}.  
\end{proof}

\section{Application to the Lusin approximation of an absolutely continuous curve}
\label{s:5}

In a recent paper, E. Le Donne and G.~Speight prove the following result 
(\cite[Theorem 1.2]{LeDonne-Speight}).
\begin{prop}[Le Donne--Speight]\label{p:speight}
Let $\mathbb{G}$ be a Carnot group of step 2 and consider 
 a horizontal curve
$\gamma:[a,b]\to \mathbb{G}$. For any $\eps>0$, there exist $K\subset [a,b]$ 
and a $\mathcal{C}^1_H$-curve $\gamma_1:[a,b]\to\mathbb{G}$ such that $\mathcal{L}([a,b]\setminus K)<\eps$ and $\gamma=\gamma_1$ on $K$. 
\end{prop}

In the case in which $\mathbb{G}$ is equal to the $n$-th Heisenberg group $\mathbb{H}_n$, 
such result had already been proved in  
\cite[Theorem 2.2]{Spe} (see also \cite[Corollary 3.8]{zimmerman}).
In \cite{Spe} G.~Speight also identifies a horizontal curve on the Engel group such that the statement of Proposition~\ref{p:speight} is not satisfied (\cite[Theorem 3.2]{Spe}).

The name ``Lusin approximation'' for the property stated in Proposition~\ref{p:speight}
comes from the use of the classical theorem of Lusin \cite{Lusin} in the proof. Let us sketch a proof when $\mathbb{G}$ is replaced by a vector space $\R^{n}$. The derivative $\dot{\gamma}$ of an absolutely continuous curve $\gamma$ is an integrable function. Lusin's theorem states that 
$\dot \gamma$  coincides with a continuous vector-valued function $X: K\to \R^{n}$
on a set $K$ of measure arbitrarily close to $b-a$. 
Thanks to the inner continuity of the Lebesgue measure, one can assume that $K$ is compact. Moreover $K$ can be chosen so that the Whitney condition is satisfied by $(\gamma|_K,X)$ on $K$. This is a consequence of the mean value inequality
\begin{align}\label{domination}
\|\gamma(x+h)-\gamma(x)-h\dot{\gamma}(x)\|\leq o(h),
\end{align}
where $o(h)$ depends on $x\in K$. By usual arguments of measure theory, inequality~\eqref{domination} can be made uniform with respect to $x$ if one slightly reduces the measure of $K$. The (classical) Whitney extension theorem provides a $\mathcal{C}^1$-curve $\gamma_1$ defined on $[a,b]$ with $\gamma_1=\gamma$ and $\dot{\gamma_1}=X$ on $K$.

The proof in \cite{LeDonne-Speight} (and also in \cite{Spe}) follows the same scheme as the one sketched above. 
We show here below how the same scheme can be adapted 
to any pliable Carnot group. The fact that all Carnot groups of step 2 are pliable and that not all pliable 
Carnot groups are of step 1 or 2 
is proved in the next section (Theorem~\ref{t:step2} and Proposition~\ref{p:every_step}), so that our paper actually provides a nontrivial generalization of Proposition~\ref{p:speight}.
The novelty of our approach with respect to those in \cite{LeDonne-Speight,Spe, zimmerman} is 
to replace the classical Rademacher differentiablility theorem for Lipschitz or absolutely continuous curves from $\R$ to $\R^n$ by the more adapted Pansu--Rademacher theorem.

\begin{prop}[Lusin approximation of a horizontal curve]\label{our_Lusin_appr}
Let $\mathbb{G}$ be a pliable Carnot group and $\gamma:[a,b]\to \mathbb{G}$ be a horizontal curve. Then for any $\eps>0$ there exist $K\subset [a,b]$ with $\mathcal{L}([a,b]\setminus K)< \eps$ and a curve $\gamma_1:[a,b]\to \mathbb{G}$ of class $\mathcal{C}^1_H$ such that the curves $\gamma$ and $\gamma_1$ coincide on $K$. 
\end{prop}
\begin{proof}
We are going to prove that for any $\eps>0$ there exists a compact set $K\subset [a,b]$ with $\mathcal{L}([a,b]\setminus K)<\eps$ such the three following conditions are satisfied:
\begin{enumerate}
\item\label{(1)} $\dot\gamma(t)$ exists and it is a horizontal vector at every $t\in K$;
\item\label{(2)} $\dot\gamma|_K$ is uniformly continuous;
\item\label{(3)} For every $\eps>0$ there exists $\eta>0$ such that, for every $t\in K$
and $|h|\leq \eta$ with $t+h\in [a,b]$, it holds
$d_\mathbb{G}(\gamma(t+h),\gamma(t)\cdot\exp(h\dot\gamma(t)))\leq \eta \eps$. 
\end{enumerate}
With these conditions the $\mathcal{C}^1_H$-Whitney condition holds for $(\gamma,\dot\gamma|_K)$ on $K$. Since $\mathbb{G}$ is pliable, according to Theorem \ref{thm-sufficient} the $\mathcal{C}^1_H$ extension property holds for $(\R,\mathbb{G})$, yielding $\gamma_1$ as in the statement of Proposition \ref{our_Lusin_appr}. 

\emph{Case 1: $\gamma$ is Lipschitz continuous.}
Let $\gamma$ be a Lipschitz curve from $[a,b]$ to $\mathbb{G}$. 
The Pansu--Rademacher theorem (Proposition \ref{PR_them}) states that there exists $A\subset [a,b]$ of full measure such that, for any $t\in A$, the curve $\gamma$ admits a derivative at $t$ and it holds
$$d_\mathbb{G}(\gamma(t+h),\gamma(t)\cdot \exp(h\dot\gamma(t)))=o(h),$$
as $h$ goes to zero.
Let $\eps$ be positive. By Lusin's theorem, one can restrict $A$ to a compact set $K_1\subset A$ such that $t\mapsto \dot{\gamma}(t)$ is uniformly continuous on $K_1$ and $\mathcal{L}(A\setminus K_1)< \eps/2$. Moreover by classical arguments of measure theory, the functions $h\mapsto |h|^{-1}d_\mathbb{G}(\gamma(t+h),\gamma(t)\cdot \exp(h\dot\gamma(t_0)))$ can be bounded by a function that is $o(1)$ as $h$ goes to zero, uniformly in $t$ on some compact set $K_2$ with $\mathcal{L}(A\setminus K_2)< \eps/2$ . In other words for every $\eps>0$ there exists $\eta$ such that for $t\in K_2$ and $h\in[t-\eta,t+\eta]$ it holds 
$$d_\mathbb{G}\left(\gamma(t+h),\gamma(t)\cdot \exp(h\dot\gamma(t))\right)\leq \eps |h|.$$
With $K=K_1\cap K_2$, the three conditions \eqref{(1)}, \eqref{(2)},  \eqref{(3)} listed above hold true.

\emph{Case 2: $\gamma$ general horizontal curve.}
Let $\gamma$ be absolutely continuous on $[a,b]$. It admits a pathlength parametrisation, i.e., there exists a Lipschitz continuous curve $\varphi:[0,T]\to \mathbb{G}$ and a function $F:[a,b]\to[0,T]$, absolutely continuous and non-decreasing, such that $\gamma=\varphi\circ F$. Moreover $\dot\varphi$ has norm $1$ at almost every time. As $F$ is absolutely continuous, for every $\eps>0$ there exists $\eta$ such that, for any measurable $K$, the inequality $\mathcal{L}([0,T]\setminus K)< \eta$ implies $\mathcal{L}([a,b]\setminus F^{-1}(K))< \eps$. 

Let $\eps$ be positive and let $\eta$ be a number corresponding to $\eps/2$ in the previous sentence. Applying to $F$ the scheme of proof sketched after Proposition \ref{p:speight} for $n=1$, there exists a compact set $K_F\subset [a,b]$ with $\mathcal{L}([a,b]\setminus K_F)< \eps/2$ such that $F$ is differentiable with a continuous derivative on $K_F$ and the bound in the mean value inequality is uniform on $K_F$.  For the Lipschitz curve $\varphi$ and for every $\eta>0$, Case 1 provides a compact set $K_\varphi\subset [0,T]$ with the listed properties with $\eps/2$ in place $\eps$.

Let $K$ be the compact $K_F\cap F^{-1}(K_\varphi)$ and note that $\mathcal{L}([a,b]\setminus K)<\eps$. For $t\in K$ it holds
$$|F(t+h)-F(t)-hF'(t)|=o(h)$$
and
$$d_\mathbb{G}\left(\varphi(F(t)+H),\varphi(F(t))\cdot\exp(H\dot{\varphi}(F(t)))\right)=o(H),$$
as $h$ and $H$ go to zero, uniformly with respect to $t\in K$. We also know that $t\mapsto F'(t)$ and $t\mapsto \dot{\varphi}(F(t))\in\mathfrak{G}_H$ exist and are continuous on $K$. It is a simple exercise to compose the two Taylor expansions and obtain the wanted conditions for $\gamma=\varphi\circ F$. Note that the derivative of $\gamma$ on $K$ is $F'(t)\dot{\varphi}(F(t))$, which is continuous on $K$.\end{proof}

\begin{remark}
A set $E\subset \R^n$ is said $1$-countably rectifiable if there exists a countable family of Lipschitz curves $f_k:\R\to \R^n$ such that 
$$\mathcal{H}^1\left(E\setminus\bigcup_k f_k(\R)\right)=0.$$
The usual Lusin approximation of curves in $\R^n$ permits one to replace Lipschitz by $\mathcal{C}^1$ in this classical definition of rectifiability. When $\R^n$ is replaced by a pliable Carnot group the  two definitions still make sense and, according to Proposition \ref{our_Lusin_appr}, are still equivalent. Rectifiability in metric spaces and Carnot groups is a very active research topic in geometric measure theory (see \cite{LeDonne-Speight} for references). 
\end{remark}

\section{Conditions ensuring pliability}
\label{s:6}

The goal of this section is to identify conditions ensuring that $\mathbb{G}$ is pliable. 
Let us first focus on the pliability of the zero vector.

\begin{proposition}\label{p:zero-pliable} 
For every Carnot group $\mathbb{G}$, the vector $0\in\mathfrak{G}$ is pliable.
\end{proposition}
\begin{proof}
According to Proposition~\ref{p:wp-vwp}, we should prove that for every $\eps>0$ 
the set 
$$\{\beta(1) \in \mathbb{G}\mid \beta\in \mathcal{C}^1_H([0,1],\mathbb{G}),\;\|\dot{\beta}\|_{\infty,\mathfrak{G}_H}<\eps,\; (\beta,\dot{\beta})(0)=(0_\mathbb{G},0)\}$$
is a neighborhood of $0_\mathbb{G}$ in $\mathbb{G}$.

Recall that there exist $k\in \N$, $V_1,\dots,V_k\in \mathfrak{G}_H$ and $t_1,\dots,t_k>0$ such that 
the map
$$\phi:(\tau_1,\dots,\tau_k)\mapsto e^{\tau_k V_k} \circ \dots \circ e^{\tau_1 V_1}(0_\mathbb{G})$$
has rank equal to $\mathrm{dim}(\mathbb{G})$ at $(\tau_1,\dots,\tau_k)=(t_1,\dots,t_k)$ 
and satisfies $\phi(t_1,\dots,t_k)=0_\mathbb{G}$
(see \cite{Sussmann76}).
Notice that for every $\nu>0$, the function 
\begin{align*}
\phi_\nu:(\tau_1,\dots,\tau_k)\mapsto &e^{\nu \tau_k V_k} \circ \dots \circ e^{\nu \tau_1 V_1}(0_\mathbb{G})\\
&=e^{\frac{\nu^2 \tau_k}{\nu}  V_k} \circ \dots \circ e^{\nu^2\frac{\tau_1}{\nu}  V_1}(0_\mathbb{G})=\delta_{\nu^2}\left(\phi\left(\frac{\tau_1}\nu,\dots,\frac{\tau_k}\nu\right)\right)
\end{align*}
has also rank equal to $\mathrm{dim}(\mathbb{G})$ at $(\tau_1,\dots,\tau_k)=(\nu t_1,\dots,\nu t_k)$ 
and satisfies $\phi_\nu(\nu t_1,\dots,\nu t_k)=0_\mathbb{G}$.
Hence, up to replacing 
$t_j$ by $\nu t_j$ 
and 
$V_j$ by $\nu^2 V_j$ for $j=1,\dots, k$ and $\nu$ small enough, we can assume that
$t_1+\dots+t_k<1$ and 
$\|V_j\|_{\mathfrak{G}_H}<\eps$ for $j=1,\dots,k$. 

Let $O$ be a neighborhood of $(t_1,\dots,t_k)$ such that for every $(\tau_1,\dots,\tau_k)\in O$ 
we have $\tau_1,\dots,\tau_k>0$ and $\tau_1+\dots+\tau_k<1
$. 
Notice that $\{\phi(\tau_1,\dots,\tau_k)\mid (\tau_1,\dots,\tau_k)\in O\}$ is a neighborhood of $0_\mathbb{G}$ in $\mathbb{G}$.

We complete the proof of the proposition by constructing, for every $\tau=(\tau_1,\dots,\tau_k)\in O$ a curve $\beta_{\tau} \in \mathcal{C}^1_H([0,1],\mathbb{G})$ such that 
\begin{equation}\label{betatau}
\|\dot{\beta}_\tau\|_{\infty,\mathfrak{G}_H}<\eps,\quad (\beta_\tau,\dot{\beta}_\tau)(0)=(0_\mathbb{G},0),\quad \beta_\tau(1)=\phi(\tau). 
\end{equation}

For every $X\in \mathfrak{G}_H$, $p\in \mathbb{G}$ and $r>0$ let us exhibit a curve $\gamma\in \mathcal{C}^1_H([0,r],\mathbb{G})$ such that $\gamma(0)=\gamma(r)=p$, $\dot \gamma(0)=0$, $\dot \gamma(r)=X$, and $\|\dot \gamma\|_{\infty,\mathfrak{G}_H}=\|X\|_{\mathfrak{G}_H}$. 
The curve $\gamma$ can be constructed  by imposing $\dot\gamma(r/2)=-X/2$ and by extending 
$\dot{\gamma}$ on $[0,r/2]$ and $[r/2,r]$ by convex interpolation. 
 It is also possible to reverse such a curve by transformation \ref{T4} and connect 
 on any segment $[0,r]$ 
 the point-with-velocity $(p,X)$ with the point-with-velocity $(p,0)$ by a $\mathcal{C}^1_H$ curve $\gamma$ respecting moreover $\|\dot\gamma\|_{\infty,\mathfrak{G}_H}= \|X\|_{\mathfrak{G}_H}$. 
 Finally just concatenating (transformation \ref{T5}) curves of this type it is possible, for every $r>0$,  to connect $(p,X)$ and $(p,Y)$ on  $[0,r]$ with a curve $\gamma_{r,X,Y}\in \mathcal{C}^1_H([0,r],\mathbb{G})$ with 
$\|\dot\gamma_{r,X,Y}\|_{\infty,\mathfrak{G}_H}= \max(\|X\|_{\mathfrak{G}_H},\|Y\|_{\mathfrak{G}_H})$.

We then construct $\beta_\tau$ as follows: we fix $r=(1-\sum_{j=1}^k\tau_j)/k$, we impose $\beta_\tau(0)=0_\mathbb{G}$  and 
we define $\dot\beta_\tau$ to be the concatenation of the following $2k$ continuous curves in $\mathfrak{G}_H$: first take $\dot\gamma_{r,0,V_1}$, then the constant equal to $V_1$ for a time $\tau_1$, then 
$\dot\gamma_{r,V_1,V_2}$, then the constant equal to $V_2$ for a time $\tau_2$, and so on up to
$\dot\gamma_{r,V_{k-1},V_k}$ and finally 
the constant equal to $V_k$ for a time $\tau_k$.
By construction,  $\beta_\tau\in \mathcal{C}^1_H([0,1],\mathbb{G})$ and satisfies \eqref{betatau}.
\end{proof}

\begin{remark}\label{zero-not-unif}
Let us show that, as a consequence of the previous proposition, pliability and local uniform pliability are not equivalent properties (albeit we know from Proposition~\ref{p=>up} that pliability of all horizontal vectors is equivalent to local uniform pliability of all horizontal vectors). 

Recall that local uniform pliability of a horizontal vector $X$ implies pliability of all horizontal vectors in a neighborhood of $X$ (cf.~Definition~\ref{def:up}). Therefore, if $0$ is locally uniformly pliable for a Carnot group $\mathbb{G}$ then every horizontal vector of $\mathfrak{G}$ is pliable (Remark~\ref{r:up-homogeneous}). 
Hence $0$ cannot be locally uniformly pliable if $G$ is not pliable. 
The remark is concluded by recalling that non-pliable Carnot groups exist (see Examples~\ref{e:Engel} and \ref{e:superEngel}).
\end{remark}

Let $\mathbb{G}$ be a 
Carnot group and let $X_1,\dots,X_m$ be an orthonormal basis of $\mathfrak{G}_{H}$. 
Let us consider the control system in $\mathbb{G}\times \R^m$ given by
\begin{equation}\label{CS}
\left\{
\begin{aligned}
\dot\gamma&=\sum_{i=1}^{m} u_i X_i(\gamma),\\
\dot{u}&=v, 
\end{aligned}
\right.
\end{equation}
where both $u=(u_1,\dots,u_m)$ and the control $v=(v_1,\dots,v_m)$ vary in $\R^m$.

Let us rewrite $x=(\gamma,u)$, 
\begin{equation*}\label{Fis}
F_0(x)=\left(\begin{array}{c}\sum_{i=1}^{m} u_i X_i(\gamma)\\0\end{array}\right), \qquad F_i(x)=\left(\begin{array}{c}0\\
e_i\end{array}\right)\quad \mbox{ for }i=1,\dots,m,
\end{equation*}
 where $e_1,\dots,e_m$ denotes the canonical basis of $\R^m$. 
System~\eqref{CS} can then be rewritten as
\begin{equation}\label{CS2}
\dot x=F_0(x)+\sum_{i=1}^m v_i F_i(x).
\end{equation}

For every $\bar u\in \R^m$, let $\mathcal{F}_{\bar u}:L^1([0,1],\R^m)\to  \mathbb{G}\times\R^m$ be the 
endpoint map at time $1$
  for system \eqref{CS2} with initial condition $(0_\mathbb{G},\bar u)$.
Notice that if $x(\cdot)=(\gamma(\cdot),u(\cdot))$ is a solution of  \eqref{CS2} 
with initial condition $(0_\mathbb{G},\bar u)$
corresponding to a control $v\in L^1([0,1],\R^m)$, then $\gamma\in \mathcal{C}^1_H([0,1],\mathbb{G})$ and 
$\|\dot \gamma-\sum_{i=1}^m \bar u_i X_i\|_{\infty,\mathfrak{G}_H}\leq \|v\|_1$. 

We can then state the following criterium for pliability. 
\begin{proposition}
If the map $\mathcal{F}_{\bar u}:L^1([0,1],\R^m)\to  \mathbb{G}\times\R^m$ is open at $0$, then  the horizontal vector $\sum_{i=1}^m \bar u_i X_i$ is pliable. 
\end{proposition}
As a consequence, if the restriction of   $\mathcal{F}_{\bar u}$ to $L^\infty([0,1],\R^m)$ is open at $0$, when  the $L^\infty$ topology is considered on $L^\infty([0,1],\R^m)$, then $\sum_{i=1}^m \bar u_i X_i$ is pliable. We deduce the following property: \emph{if a straight curve is not pliable, then it admits an abnormal lift in $T^*\mathbb{G}$.  Indeed,  if a horizontal vector $\sum_{i=1}^m \bar u_i X_i$ is not pliable, then the differential of $\mathcal{F}_{\bar u}|_{L^\infty([0,1],\R^m)}$} at $0$ must be singular.
Hence (see, for instance, \cite[Section 20.3]{AgrSac} or \cite[Proposition 5.3.3]{Trelat}), there exist $p_\gamma:[0,1]\to T^*\mathbb{G}$ and $p_u:[0,1]\to (\R^m)^*$ with $(p_\gamma,p_u)\ne 0$ 
such that 
\begin{align}
\dot p_\gamma(t)&=-\frac{\partial}{\partial \gamma}H(\gamma(t),\bar u,p_\gamma(t),p_u(t),0),\label{H-lg}\\
 \dot p_u(t)&=-\frac{\partial}{\partial u}H(\gamma(t),\bar u,p_\gamma(t),p_u(t),0),\label{H-lu}\\
0&= \frac{\partial}{\partial v}H(\gamma(t),\bar u,p_\gamma(t),p_u(t),0),\label{H-v}
\end{align}
for $t\in[0,1]$, where $\gamma(t)=\exp(t \sum_{i=1}^m \bar u_i X_i)$ and 
$$H(\gamma,u,p_\gamma,p_u,v)=p_\gamma\sum_{i=1}^m u_i X_i(\gamma)+ p_u v.$$
From \eqref{H-v} it follows that $p_{u}(t)= 0$ for all $t\in[0,1]$. Equation \eqref{H-lu} then implies that $p_\gamma(t) X_i(\gamma(t))=0$ for every $i=1,\dots,m$ and every $t\in[0,1]$. Moreover, $p_\gamma$ must be different from zero. Comparing \eqref{eq-abn} and \eqref{H-lg}, it follows that $p_\gamma$ is an abnormal path. 

The control literature proposes several criteria  for testing the  
openness at $0$ of an endpoint map of the type $\mathcal{F}_{\bar u}|_{L^\infty([0,1],\R^m)}$.
The test presented here below, taken from \cite{BianchiniStefani}, generalizes previous criteria obtained in  \cite{Hermes82} and \cite{Sussmann87}.

\begin{theorem}[{Bianchini and Stefani \cite[Corollary 1.2]{BianchiniStefani}}]\label{HermesSussmann}
Let $M$ be a $\mathcal{C}^\infty$ manifold and 
$V_0,V_1,\dots,V_m$ be $\mathcal{C}^\infty$ vector fields on  $M$. 
Assume that the family of vector fields $\mathcal{J}=\{\mathrm{ad}^k_{V_0}V_j\mid k\geq 0,\;j=1,\dots,m\}$ is Lie bracket generating. 
Denote by $\mathcal{H}$ the iterated brackets of elements in $\mathcal{J}$ and recall that the length of an element of $\mathcal{H}$ is the sum of the number of times that each of the elements $V_0,\dots,V_m$ appears in its expression.  
Assume that every element of $\mathcal{H}$ in whose expression  each of the vector fields $V_1,\dots,V_m$ appears an even number of times is equal, at every $q\in M$, to the linear combination of elements of $\mathcal{H}$ of smaller length, evaluated at $q$. 
Fix $q_0\in M$ and a neighborhood $\Omega$ of $0$ in $\R^m$.
Let $\mathcal U\subset L^\infty([0,1],\Omega)$ 
be the set of those controls $v$ such that 
the solution of $\dot q=V_0(q)+\sum_{i=1}^m v_i V_i(q)$, $q(0)=q_0$, is defined up to time $1$ and denote by $\Phi(v)$ the endpoint $q(1)$ of such a solution. 
Then $\Phi(\mathcal{U})$ 
 is a neighborhood of $e^{V_0}(q_0)$. 
\end{theorem}

The following two results show how to apply Theorem~\ref{HermesSussmann} to guarantee that a Carnot group $\mathbb{G}$ is pliable and, hence, that $(\R,\mathbb{G})$ has the $\mathcal{C}^1_H$ extension property. 
\begin{theorem}\label{t:step2}
Let $\mathbb{G}$ be a Carnot group of step $2$. Then $\mathbb{G}$ is pliable and $(\R,\mathbb{G})$ has the $\mathcal{C}^1_H$ extension property.
\end{theorem}
\begin{proof}
We are going to apply Theorem~\ref{HermesSussmann}
in order to 
prove that for every horizontal vector $\sum_{i=1}^m u_iX_i$ the 
endpoint map $\mathcal{F}_{u}:L^\infty([0,1],\R^m)\to  \mathbb{G}\times\R^m$
is open at zero.

Notice that
$$[F_0,F_i](\gamma,w)=-\begin{pmatrix}X_i(\gamma)\\ 0\end{pmatrix},\qquad i=1,\dots,m,$$
and
$$[F_0,[F_0,F_i]](\gamma,w)=\begin{pmatrix}\sum_{j=1}^m w_j[X_i,X_j](\gamma)\\ 0\end{pmatrix},\qquad i=1,\dots,m.$$
Moreover for every  $i,j=1,\dots,m$, 
$$[[F_0,F_i],F_j]=0,\ [[F_0,F_i],[F_0,F_j]](\gamma,w)=\begin{pmatrix}[X_i,X_j](\gamma)\\ 0\end{pmatrix},$$
and all other Lie bracket in and between elements of
$\mathcal{J}=\{{\rm ad}^k_{F_0}F_i\mid k\geq 0,\;i=1,\dots,m\}$
is zero since $\mathbb{G}$ is of step 2.

In particular all Lie brackets between elements of $\mathcal{J}$ in which each of the vector fields  $F_1,\dots,F_m$ appears an even number of times is zero. 

According to Theorem~\ref{HermesSussmann},  we are left to prove that 
$\mathcal{J}$ is Lie bracket generating.
This is clearly true, since
$$\mathrm{Span}\{F_i(\gamma,w),[F_0,F_i](\gamma,w),[[F_0,F_i],[F_0,F_j]](\gamma,w)\mid i,j=1,\dots,m\}$$
is equal to $T_{(\gamma,w)}(\mathbb{G}\times \R^m)$ for every $(\gamma,w)\in \mathbb{G}\times\R^m$.
\end{proof}

We conclude the paper by showing how to construct pliable Carnot groups of arbitrarily large step. 

\begin{proposition}\label{p:every_step}
For every  $s\ge 1$ there exists a pliable Carnot group  of step $s$. 
\end{proposition}
\begin{proof}
Fix $s\ge 1$ and consider the free nilpotent stratified Lie algebra $\mathcal{A}$ of step $s$ generated by $s$ elements $Z_1,\dots,Z_s$. 

For every $i=1,\dots,s$, denote by $I_i$ the ideal of $\mathcal{A}$ generated by $Z_i$ and by $J_i$ the ideal $[I_i,I_i]$. Then $J=\oplus_{i=1}^s J_i$ is also an ideal of $\mathcal{A}$.

 Then the factor algebra $\mathfrak{G}=\mathcal{A}/J$ is nilpotent and inherits the stratification of $\mathcal{A}$.  
 Denote by  $\mathbb{G}$ the Carnot group
 generated by $\mathfrak{G}$.
 Let  $X_1,\dots,X_s$ be the elements of $\mathfrak{G}_H$ obtained by projecting  $Z_1,\dots,Z_s$. 
  By construction, every bracket of $X_1,\dots,X_s$ in $\mathfrak{G}$ in which at least one  of the $X_i$'s appears more than once is zero. Moreover,  $\mathfrak{G}$ has step $s$, since $[X_1,[X_2,[\,\cdots,X_s],\cdots]$ is different from zero. 
  
 Let us now apply  
  Theorem~\ref{HermesSussmann} in order to 
prove that for every $X\in\mathfrak{G}_H$ the 
endpoint map $\mathcal{F}_{u}:L^\infty([0,1],\R^s)\to  \mathbb{G}\times\R^s$
is open at zero, where $u\in\R^s$ is such that $X=\sum_{i=1}^s u_i X_i$. 
 
Following the same computations as in the proof  of Theorem~\ref{t:step2}, 
$$\mathrm{ad}_{F_0}^{k+1} F_i(\gamma,u)=\begin{pmatrix}\mathrm{ad}_{X}^k X_i(\gamma)\\ 0\end{pmatrix},\qquad k\geq0,\ i=1,\dots,s.$$
In particular the family $\mathcal{J}=\{{\rm ad}^k_{F_0}F_i\mid k\geq 0,\;i=1,\dots,s\}$ is Lie bracket generating.  

Moreover,  every Lie bracket of elements of $\widehat{\mathcal{J}}=\{\mathrm{ad}_{F_0}^{k+1} F_i\mid k\geq 0,\;i=1,\dots,s\}$ in which at least one  of the elements $F_1,\dots,F_s$ appears more than once is zero.

Consider now a Lie bracket $W$ between $h\geq 2$ elements of $\mathcal{J}$. Let $k_1,\dots,k_s$ be the number of times in which each of the elements $F_1,\dots,F_s$ appears in $W$. 
Let us prove by induction on $h$ that $W$ is the linear combination of brackets between elements of $\widehat{\mathcal{J}}$ 
in which each $F_i$ appears $k_i$ times, $i=1,\dots,s$.
Consider the case $h=2$. 
Any bracket of the type $[\mathrm{ad}_{F_0}^{k} F_i,F_j]$, $k\ge 0$, $i,j=1,\dots,s$,
is the linear combination of brackets between elements of $\widehat{\mathcal{J}}$ 
in which $F_i$ and $F_j$ appear once, as it can easily be proved by induction on $k$, thanks to the Jacobi identity. The induction step on $h$ also follows directly from the Jacobi identity. 

We can therefore conclude that every Lie bracket of elements of $\mathcal{J}$ in which at least one  of the elements $F_1,\dots,F_s$ appears more than once is zero. 
This implies in particular that the hypotheses of Theorem~\ref{HermesSussmann} are satisfied, concluding the proof that $\mathbb{G}$ is pliable. 
\end{proof}

\noindent{\bf Acknowledgment.}
We warmly thank Fr\'ed\'eric Chapoton and Gw\'ena\"el Massuyeau for the suggestions leading us to Proposition~\ref{p:every_step}.
We are also grateful to Artem Kozhevnikov, Dario Prandi, Luca Rizzi and Andrei Agrachev 
for many stimulating discussions.
This work has been initiated during the IHP trimester ``Geometry, analysis and dynamics on sub-Riemannian manifolds'' and we wish to thank the Institut Henri Poincar\'e and the Fondation Sciences Math\'ematiques de Paris for the welcoming working conditions.

The second author has been supported by the European
Research Council, ERC StG 2009 ``GeCoMethods'', contract number 239748, by the Grant ANR-15-CE40-0018 of the ANR and by the FMJH Program Gaspard Monge in optimization and operation research.

\bibliographystyle{abbrv}
\bibliography{biblio}

\end{document}